\documentclass[oneside,english]{amsart}
\usepackage[T1]{fontenc}
\usepackage[latin1]{inputenc}
\usepackage{graphicx}
\usepackage{amssymb}

\makeatletter
 \theoremstyle{plain}    
 \newtheorem{thm}{Theorem}[section]
 \numberwithin{equation}{section} 
 \numberwithin{figure}{section} 
 \theoremstyle{plain}
 \theoremstyle{plain}    
 \newtheorem{lem}[thm]{Lemma} 
 \theoremstyle{plain}    
 \newtheorem{prop}[thm]{Proposition} 
 \theoremstyle{plain}    
 \newtheorem{cor}[thm]{Corollary} 

\usepackage{fancyhdr}
\usepackage{babel}
\makeatother
\begin{document}

\title{Intersections of Schubert Varieties and Eigenvalue Inequalities in
an Arbitrary Finite Factor}

\author{H. Bercovici, B. Collins, K. Dykema, W. S. Li, and D. Timotin}


\thanks{HB, KD, and WSL were supported in part by grants from the National
Science Foundation. BC was supported in part by an NSERC grant.}

\address{HB: Department of Mathematics, Indiana University, Bloomington, IN
47405, USA}

\email{bercovic@indiana.edu}

\address{BC: Department of Mathematics and Statistics, University of Ottawa,
Ottawa, Ontario K1N 6N5 Canada, and CNRS, Department of Mathematics,
Universit\'e Claude Bernard, Lyon 1, Lyon, France}

\email{bcollins@uottawa.edu}

\address{KD: Department of Mathematics, Texas A\&M University, College Station,
TX 77843--3368, USA}

\email{kdykema@math.tamu.edu}

\address{WSL: School of Mathematics, Georgia Institute of Technology, Atlanta,
GA 30332-1060, USA}

\email{li@math.gatech.edu}

\address{DT: Simion Stoilow Institute of Mathematics of the Romanian Academy,
PO Box 1-764, Bucharest 014700, Romania}

\email{Dan.Timotin@imar.ro}

\begin{abstract}
It is known that the eigenvalues of selfadjoint elements $a,b,c$
with $a+b+c=0$ in the factor $\mathcal{R}^{\omega}$ are characterized
by a system of inequalities analogous to the classical Horn inequalities
of linear algebra. We prove that these inequalities are in fact true
for elements of an arbitrary finite factor. In particular, if $x,y,z$
are selfadjoint elements of such a factor and $x+y+z=0$, then there
exist selfadjoint $a,b,c\in\mathcal{R}^{\omega}$ such that $a+b+c=0$
and $a$ (respectively, $b,c$) has the same eigenvalues as $x$ (respectively,
$y,z$). A (`complete') matricial form of this result is known to
imply an affirmative answer to an embedding question formulated by
Connes.

The main difficulty in our argument is the proof that certain generalized
Schubert cells (consisting of projections of fixed trace) have nonempty
intersection. In finite dimensions, this follows from classical intersection
theory. Our approach is to exhibit an actual element in this intersection,
given by a formula which applies generically regardless of the algebra
(or of the dimension of the underlying space). This argument requires
a good understanding of the combinatorial structure of honeycombs,
and it seems to be new even in finite dimensions.
\end{abstract}
\maketitle

\section*{Introduction}

Assume that $A,B,C$ are complex selfadjoint $n\times n$ matrices,
and $A+B+C=0$. A. Horn proposed in \cite{horn} the question of characterizing
the possible eigenvalues of these matrices, and indeed he conjectured
an answer which was eventually proved correct due to efforts of A.
Klyachko \cite{klyach} and A. Knutson and T. Tao \cite{KT}. To explain
this characterization, list the eigenvalues of $A$, repeated according
to multiplicity, in nonincreasing order\[
\lambda_{A}(1)\ge\lambda_{A}(2)\ge\cdots\ge\lambda_{A}(n),\]
choose an orthonormal basis $x_{j}\in\mathbb{C}^{n}$ such that $Ax_{j}=\lambda_{A}(j)x_{j}$,
and denote by $E_{A}(j)$ the space generated by $\{ x_{1},x_{2},\dots,x_{j}\}$.
Horn's conjecture involves, in addition to the trace identity\[
\sum_{j=1}^{n}(\lambda_{A}(j)+\lambda_{B}(j)+\lambda_{C}(j))=0,\]
a collection of inequalities of the form\[
\sum_{i\in I}\lambda_{A}(i)+\sum_{j\in J}\lambda_{B}(j)+\sum_{k\in K}\lambda_{C}(k)\le0,\]
where $I,J,K\subset\{1,2,\dots,n\}$ are sets with equal cardinalities.
One way to prove such inequalities is to observe that\[
\text{Tr}(PAP+PBP+PCP)=0\]
for any orthogonal projection $P$, and to find a projection $P$
such that\[
\text{Tr}(PAP)\ge\sum_{i\in I}\lambda_{A}(i),\quad\text{Tr}(PBP)\ge\sum_{j\in J}\lambda_{B}(j),\quad\text{Tr}(PCP)\ge\sum_{k\in K}\lambda_{C}(k).\]
Now, if $I=\{ i_{1}<i_{2}<\cdots<i_{r}\}$, the first condition is
guaranteed provided that the range $M$ of $P$ has dimension $r$,
and\[
\dim(M\cap E_{A}(i_{\ell}))\ge\ell,\quad\ell=1,2,\dots,r.\]
 These conditions describe the Schubert variety $S(E_{A},I)$ determined
by the flag $\{ E_{A}(\ell)\}_{\ell=1}^{n}$ and the set $I$. Thus,
such a projection can be found provided that\[
S(E_{A},I)\cap S(E_{B},J)\cap S(E_{C},K)\ne\varnothing.\]
 Klyachko \cite{klyach} proved that the collection of all inequalities
obtained this way is sufficient to answer Horn's question, and observed
that Horn's conjectured answer would also be proved if a certain `saturation
conjecture' were true. This conjecture was proved by Knutson and Tao
\cite{KT}. (See also \cite{KTW} for a direct proof of Horn's conjecture,
and \cite{ful-bull} for a very good survey of the history of the
problem and its ramifications. Some earlier expositions are in \cite{day-tho,ful-bourb}.)

There are several infinite-dimensional analogues of the Horn problem.
One can for instance consider compact selfadjoint operators $A,B,C$
on a Hilbert space and their eigenvalues. This analogue was considered
by several authors \cite{friedland,ful-linalg}, and a complete solution
can be found in \cite{ber-li-smo} for operators such that $A,B$,
and $-C$ are positive, and \cite{BLT} for the general case. Without
going into detail, let us say that these solutions are based on an
understanding of the behavior of the Horn inequalities as the dimension
of the space tends to infinity. The analogue we are interested in
here replaces the algebra $M_{n}(\mathbb{C})$ of $n\times n$ matrices
by a finite factor. This is simply a selfadjoint algebra $\mathcal{A}$
of operators on a complex Hilbert space $H$ such that $\mathcal{A}'\cap\mathcal{A}=\mathbb{C}1_{H}$
(where $\mathcal{A}'=\{ T:AT=TA\text{ for all }A\in\mathcal{A}\}$),
$\mathcal{A}''=\mathcal{A}$, and for which there exists a linear
functional $\tau:\mathcal{A}\to\mathbb{C}$ such that $\tau(X^{*}X)=\tau(XX^{*})>0$
for all $X\in\mathcal{A}\setminus\{0\}$. The algebras $M_{n}(\mathbb{C})$
are finite factors. When $\mathcal{A}$ is an infinite dimensional
finite factor, it is called a factor of type II$_{1}$. A complete
flag in a II$_{1}$ factor $\mathcal{A}$ is a family of orthogonal
projections $\{ E(t):0\le t\le\tau(1_{H})\}$ such that $\tau(E(t))=t$,
and $E(t)\le E(s)$ for $t\le s$. For any selfadjoint operator $A\in\mathcal{A}$
there exist a nonincreasing function $\lambda_{A}:[0,\tau(1_{H})]\to\mathbb{R}$,
and a complete flag $\{ E_{A}(t):0\le t\le\tau(1_{H})\}$ such that\[
A=\int_{0}^{\tau(1_{H})}\lambda_{A}(t)\, dE_{A}(t).\]
This is basically a restatement of the spectral theorem. The function
$\lambda_{A}$ is uniquely determined at its points of continuity,
but the space $E_{A}(t)$ is not uniquely determined on the open intervals
where $\lambda_{A}$ is constant. Note that\[
\tau(A)=\int_{0}^{\tau(1_{H})}\lambda_{A}(t)\, dt,\]
and therefore we have a trace identity\[
\int_{0}^{\tau(1_{H})}(\lambda_{A}(t)+\lambda_{B}(t)+\lambda_{C}(t))\, dt=0\]
whenever $A+B+C=0$. 

Many factors of type II$_{1}$ can be approximated in a weak sense
by matrix algebras. These are the factors that embed in the ultrapower
$\mathcal{R}^{\omega}$ of the hyperfinite II$_{1}$ factor $\mathcal{R}$.
For elements in such factors one can easily prove analogues of Horn's
inequalities. More precisely, assume that $\lambda:[0,T]\to\mathbb{R}$
is a nonincreasing function. The sequence\[
\lambda^{(n)}(1)\ge\lambda^{(n)}(2)\ge\cdots\ge\lambda^{(n)}(n)\]
is defined by $\lambda^{(n)}(j)=\int_{(j-1)T/n}^{jT/n}\lambda(t)\, dt$.
The following result was proved in \cite{Ber-Li-Romega}. We use the
normalization $\tau(1_{H})=1$ for the factor $\mathcal{R}^{\omega}$.

\begin{thm}
\label{th:hyperfinite-characterization}Let $\alpha,\beta,\gamma:[0,1]\to\mathbb{R}$
be nonincreasing functions. The following are equivalent:
\begin{enumerate}
\item There exist selfadjoint operators $A,B,C\in\mathcal{R}^{\omega}$
such that $\lambda_{A}=\alpha,\lambda_{B}=\beta,$ $\lambda_{C}=\gamma,$
and $A+B+C=0$. 
\item For every integer $n\ge1$, there exist matrices $A_{n},B_{n},C_{n}\in M_{n}(\mathbb{C})$
such that $\lambda_{A_{n}}=\alpha^{(n)},\lambda_{B_{n}}=\beta^{(n)},$
$\lambda_{C_{n}}=\gamma^{(n)},$ and $A_{n}+B_{n}+C_{n}=0$.
\end{enumerate}
\end{thm}
Note that condition (2) above requires, in addition to the trace identity,
an infinite (and infinitely redundant) collection of Horn inequalities.
We will show that these inequalities are in fact satisfied in any
factor of type II$_{1}$.

\begin{thm}
\label{th:characterization-in-any-factor}Given a factor $\mathcal{A}$
of type II$_{1}$, selfadjoint elements $A,B,C\in\mathcal{A}$ such
that $A+B+C=0$, and an integer $n\ge1$, there exist matrices $A_{n},B_{n},C_{n}\in M_{n}(\mathbb{C})$
such that $\lambda_{A_{n}}=\lambda_{A}^{(n)},\lambda_{B_{n}}=\lambda_{B}^{(n)},$
$\lambda_{C_{n}}=\lambda_{C}^{(n)}$, and $A_{n}+B_{n}+C_{n}=0$.
\end{thm}
The proof of the relevant inequalities relies, as in finite dimensions,
on finding projections with prescribed intersection properties. In
order to state our main result in this direction we need a more precise
description of the Horn inequalities. Assume that the subsets $I=\{ i_{1}<i_{2}<\cdots<i_{r}\}$,
$J=\{ j_{1}<j_{2}<\cdots<j_{r}\}$, and $K=\{ k_{1}<k_{2}<\cdots<k_{r}\}$
of $\{1,2,\dots,n\}$ satisfy the identity\[
\sum_{\ell=1}^{r}[(i_{\ell}-\ell)+(j_{\ell}-\ell)+(k_{\ell}-\ell)]=2r(n-r).\]
One associates to these sets a nonnegative integer $c_{IJK}$, called
the Littlewood-Richardson coefficient. The sets $I,J,K$ yield an
eigenvalue inequality in Horn's conjecture if $c_{IJK}\ne0$. Moreover,
as shown by P. Belkale \cite{belkale}, the inequalities corresponding
with $c_{IJK}>1$ are in fact redundant. Thus, the preceding theorem
follows from the next result.

\begin{thm}
\label{th:the-Horn-inequalities-in-any-A}Given a factor $\mathcal{A}$
of type II$_{1}$, selfadjoint elements $A,B,C\in\mathcal{A}$ such
that $A+B+C=0$, an integer $n\ge1$, and sets $I,J,K\subset\{1,2,\dots,n\}$
such that $c_{IJK}=1$, we have\[
\sum_{i\in I}\lambda_{A}^{(n)}(i)+\sum_{j\in J}\lambda_{B}^{(n)}(j)+\sum_{k\in K}\lambda_{C}^{(n)}(k)\le0.\]

\end{thm}
This result follows from the existence of projections satisfying specific
intersection requirements. Before stating our result in this direction,
we need to specify a notion of genericity. Fix a finite factor $\mathcal{A}$
with trace normalized so that $\tau(1)=n$. We will deal with flags
of projections with integer dimensions, i.e., with collections\[
\mathcal{E}=\{0=E_{0}<E_{1}<\cdots<E_{n}=1\}\]
of orthogonal projections in $\mathcal{A}$ such that $\tau(E_{j})=j$
for all $j$. Given such a flag and a unitary operator $U\in\mathcal{A}$,
the projections $U\mathcal{E}U^{*}=\{ UE_{j}U^{*}:0\le j\le n\}$
form another flag. In fact, all flags with integer dimensions are
obtained this way. A statement about a collection of three flags $\{\mathcal{E},\mathcal{F},\mathcal{G}\}$
will be said to hold generically (or for generic flags) if it holds
for the flags $\{ U\mathcal{E}U^{*},V\mathcal{E}V^{*},W\mathcal{E}W^{*}\}$
with $(U,V,W)$ in a norm-dense open subset of $\mathcal{U}(\mathcal{A})^{3}$
(where $\mathcal{U}(\mathcal{A})$ denotes the group of unitaries
in $\mathcal{A}$). In finite dimensions, this set of unitaries can
usually be taken to be Zariski open.

Note that flags with integer dimensions always exist if $\mathcal{A}$
is of type II$_{1}$. If $\mathcal{A}=M_{m}(\mathbb{C})$, such flags
only exist when $n$ divides $m$.

One final piece of notation. Given variables $\{ e_{j},f_{j},g_{j}:1\le j\le n\}$,
we consider the free lattice $L=L(\{ e_{j},f_{j},g_{j}:1\le j\le n\})$.
This is simply the smallest collection which contains the given variables,
and has the property that, given $p,q\in L$, the expressions $(p)\wedge(q)$
and $(p)\vee(q)$ also belong to $L$. We refer to the elements of
$L$ as \emph{lattice polynomials}. If $p$ is a lattice polynomial
and $\{ E_{j},F_{j},G_{j}:1\le j\le n\}$ is a collection of orthogonal
projections in a factor $\mathcal{A}$, we can substitute projections
for the variables of $p$ to obtain a new projection $p(\{ E_{j},F_{j},G_{j}:1\le j\le n\})$.
The lattice operations are interpreted as usual: $P\vee Q$ is the
projection onto the closed linear span of the ranges of $P$ and $Q$,
and $P\wedge Q$ is the projection onto the intersection of the ranges
of $P$ and $Q$. Note that we did not impose any algebraic relations
on $L$. When we work with flags, we can always reduce lattice polynomials
using the relations $e_{j}\wedge e_{k}=e_{\min\{ j,k\}}$ and $e_{j}\vee e_{k}=e_{\max\{ j,k\}}$.
Further manipulations are possible because the lattice of projections
in a finite factor is modular, i.e. $(P\vee Q)\wedge R=P\vee(Q\wedge R)$
provided that $P\le R$.

As in finite dimensions, the Horn inequalities follow from the intersection
result below. Given a flag $\mathcal{E}=(E_{j})_{j=0}^{n}\subset\mathcal{A}$
such that $\tau(E_{j})=j$, and a set $I=\{ i_{1}<i_{2}<\cdots<i_{r}\}\subset\{1,2,\dots,n\}$,
we denote by $S(\mathcal{E},I)$ the collection of projections $P\in\mathcal{A}$
satisfying $\tau(P)=r$ and\[
\tau(P\wedge E_{i_{\ell}})\ge\ell,\quad\ell=1,2,\dots,r.\]

\begin{thm}
\label{th:intersections-not-generic}Given subsets $I,J,K\subset\{1,2,\dots,n\}$
with cardinality $r$, and with the property that $c_{IJK}=1$, a
finite factor $\mathcal{A}$ with $\tau(1)=n$, and arbitrary flags
$\mathcal{E}=(E_{j})_{j=0}^{n}$, $\mathcal{F}=(F_{j})_{j=0}^{n},$
$\mathcal{G}=(G_{j})_{j=0}^{n}$ such that $\tau(E_{j})=\tau(F_{j})=\tau(G_{j})=j$,
the intersection\[
S(\mathcal{E},I)\cap S(\mathcal{F},J)\cap S(\mathcal{G},K)\]
is not empty.
\end{thm}
For generic flags, more is true.

\begin{thm}
\label{th:intersections-generic-case}Given subsets $I,J,K\subset\{1,2,\dots,n\}$
with cardinality $r$, and with the property that $c_{IJK}=1$, there
exists a lattice polynomial $p\in L(\{ e_{j},f_{j},g_{j}:0\le j\le n\})$
with the following property: for any finite factor $\mathcal{A}$
with $\tau(1)=n$, and for generic flags $\mathcal{E}=(E_{j})_{j=0}^{n}$,
$\mathcal{F}=(F_{j})_{j=0}^{n},$ $\mathcal{G}=(G_{j})_{j=0}^{n}$
such that $\tau(E_{j})=\tau(F_{j})=\tau(G_{j})=j$, the projection
$P=p(\mathcal{E},\mathcal{F},\mathcal{G})$ has trace $\tau(P)=r$
and, in addition\[
\tau(P\wedge E_{i})=\tau(P\wedge F_{j})=\tau(P\wedge G_{k})=\ell\]
when $i_{\ell}\le i<i_{\ell+1},j_{\ell}\le j<k_{\ell+1},k_{\ell}\le k<k_{\ell+1}$
and $\ell=0,1,\dots,r,$ where $i_{0}=j_{0}=k_{0}=0$ and $i_{r+1}=j_{r+1}=k_{r+1}=n+1$.
\end{thm}
When $\mathcal{A}=M_{n}(\mathbb{C})$, the existence and generic uniqueness
of a projection $P$ satisfying the trace conditions in the statement
is well-known. In fact $c_{IJK}$ serves as an algebraic way to count
these projections. Our argument works equally well for linear subspaces
of $\mathbb{F}^{n}$ for any field $\mathbb{F}$ (except that orthogonal
complements $1-P$ must be replaced by annihilators in the dual).
The following result is generally false when $c_{IJK}>1$.

\begin{thm}
\label{thm:intersections-field-cse}Fix a field $\mathbb{F}$, and
complete flags $\mathcal{E}=(E_{j})_{j=0}^{n}$, $\mathcal{F}=(F_{j})_{j=0}^{n},$
$\mathcal{G}=(G_{j})_{j=0}^{n}$ of subspaces in $\mathbb{F}^{n}$.
Given subsets $I,J,K\subset\{1,2,\dots,n\}$ with cardinality $r$,
and with the property that $c_{IJK}=1$, there exists a subspace $M\subset\mathbb{F}^{n}$
such that $\dim M=r$, and\[
\dim(M\cap E_{i_{\ell}})\ge\ell,\quad\dim(M\cap F_{j_{\ell}})\ge\ell,\quad\dim(M\cap G_{k_{\ell}})\ge\ell\]
for $\ell=1,2,\dots,r$.
\end{thm}
The search for projections $P$ satisfying the conclusion of Theorem
\ref{th:intersections-not-generic} is much more difficult when $c_{IJK}>1$.
One of the simplest cases of this problem is equivalent to the invariant
subspace problem relative to a II$_{1}$ factor $\mathcal{A}$; this
case was first discussed in \cite{CoDy-reduction} where an approximate
solution is found. The relative invariant subspace problem remains
open but there was spectacular progress in the work of U. Haagerup
and H. Schultz \cite{haag-sch1,haag-sch2}.

The function $\lambda_{A}$ can be defined more generally for a selfadjoint
element of a von Neumann algebra $\mathcal{A}$ endowed with a faithful,
normal trace $\tau$. The inequalities in Theorem \ref{th:the-Horn-inequalities-in-any-A}
are in fact true in this more general context. Rather than prove this
fact directly, we will embed any such von Neumann algebra in a II$_{1}$
factor, in such a way that the trace is preserved. An alternative
proof can be obtained using von Neumann's reduction theory.

The remainder of this paper is organized as follows. In Section 1
we describe an enumeration of the sets $I,J,K$ with $c_{IJK}>0$
in terms of a class of measures on the plane. This enumeration is
essentially the one indicated in \cite{KTW}; the measures we use
can be viewed as the second derivatives of hives, or the first derivatives
of honeycombs. (The fact that honeycombs enumerate Schubert intersection
problems is proved in a direct way in the appendix of \cite{buch-fulton-sat};
see also Tao's `proof without words' illustrated in \cite{vakil:geom-rule}.)
We also describe the duality observed in \cite[Remark 2 on p. 42]{KTW},
realized by inflation to a puzzle, and {*}-deflation to a dual measure.
In Section 2, we use then the puzzle characterization of rigidity
from \cite{KTW} to formulate the condition $c_{IJK}=1$ in terms
of the support of the corresponding measure $m$. This result may
be viewed as the $N=0$ version of \cite[Lemma 8]{KTW}. This characterization
is used in Section 3 to show that a measure $m$ corresponding to
sets with $c_{IJK}=1$ (also called a rigid measure) can be written
uniquely as a sum $m_{1}+m_{2}+\cdots+m_{p}$ of extremal measures,
and to introduce an order relation `$\prec$' on the set $\{ m_{j}:1\le j\le p\}$.
In Section 4 we provide an extension of the concept of clockwise overlay
from \cite{KTW}, and show that `$\prec$' provides examples of clockwise
overlays. The main results are proved in Section 5. The most important
observation is that general Schubert intersection problems can be
reduced to problems corresponding to extremal measures. The order
relation is essential here as the minimal measures $m_{j}$ (relative
to `$\prec$') must be considered first. A problem corresponding to
an extremal measure has then a dual form (obtained by taking orthogonal
complements) which is no longer extremal, except for essentially one
trivial example. Section 6 contains a number of illustrations of this
reduction procedure, including explicit expressions for the corresponding
lattice polynomials which yield the solution for generic flags. In
Section 7 we describe a particular intersection problem which is equivalent
to the invariant subspace problem relative to a II$_{1}$ factor.
In Section 8 we embed any algebra with a trace in a factor of type
II$_{1}$, and we show that projections can be moved to general position
by letting them evolve according to free unitary Brownian motion.

There has been quite a bit of recent work on the geometry and intersection
of Schubert cells. Belkale \cite{belk-horn} shows that the inductive
structure of the intersection ring of the Grassmannians can be justified
geometrically. R. Vakil \cite{vakil:geom-rule} provides an approach
to the structure of this ring by a process of flag degenerations.
He also indicates \cite{vakil-induction,vakil:geom-rule} that this
can be used in order to solve effectively all Schubert intersection
problems, at least for generic flags. More precisely, \cite[Remark 2.10]{vakil-induction}
suggests that these solutions can be found, after an appropriate parametrization,
by an application of the implicit function theorem which can be made
numerically effective. These methods would apply to arbitrary values
of $c_{IJK}$, and it is not clear that they would yield the formulas
of Theorem \ref{th:intersections-generic-case} when $c_{IJK}=1$.
The method of flag degeneration of \cite{vakil:geom-rule} depends
essentially on finite dimensionality. A prospective analogue in a
II$_{1}$ factor would require a checkerboard with a continuum of
squares, and would be played with a continuum of pieces. This kind
of game is difficult to organize, as illustrated for instance by the
cumbersome argument used in \cite{ber-li-freede}. The result proved
with so much labor in that paper is deduced very simply from our current
methods, as shown in Section 6 below.

\section{Horn Inequalities and Measures}

Fix integers $1\le r\le n$, and subsets $I,J,K\subset\{1,2,\dots,n\}$
of cardinality $r$. We will find it useful on occasion to view the
set $I$ as an increasing function $I:\{1,2,\dots,r\}\to\{1,2,\dots,n\}$,
i.e., $I=\{ I(1)<I(2)<\cdots<I(r)\}$. The results of \cite{KTW}
show that we have $c_{IJK}>0$ if and only if there exist selfadjoint
matrices $X,Y,Z\in M_{r}(\mathbb{C})$ such that $X+Y+Z=2(n-r)1_{\mathbb{C}^{r}}$
and\[
\lambda_{X}(r+1-\ell)=I(\ell)-\ell,\lambda_{Y}(r+1-\ell)=J(\ell)-\ell,\lambda_{Z}(r+1-\ell)=K(\ell)-\ell\]
for $\ell=1,2,\dots,r.$ Thus, as conjectured by Horn, such sets can
be described inductively, using Horn inequalities with fewer terms.
In other words, we have $c_{IJK}>0$ if and only if\[
\sum_{\ell=1}^{r}[(I(\ell)-\ell)+(J(\ell)-\ell)+(K(\ell)-\ell)]=2r(n-r),\]
and \[
\sum_{\ell=1}^{s}[(I(I'(\ell))-I'(\ell))+(J(J'(\ell))-J'(\ell))+(K(K'(\ell))-K'(\ell))]\ge2s(n-r)\]
whenever $s\in\{1,2,\dots,r-1\}$ and $I',J',K'\subset\{1,2,\dots,r\}$
are sets of cardinality $s$ such that $c_{I'J'K'}>0$. The last inequality
can also be written as\[
\sum_{\ell=1}^{s}[(I(I'(\ell))-\ell)+(J(J'(\ell))-\ell)+(K(K'(\ell))-\ell)]\ge2s(n-s).\]

The numbers $c_{IJK}$ can be calculated using the Littlewood-Richardson
rule which we discuss next. We use the form of the rule described
in \cite{KTW}, so we need first to describe a set of measures on
the plane. Begin by choosing three unit vectors $u,v,w$ in the plane
such that $u+v+w=0$.


\begin{center} \begin{picture}(0,0)%
\includegraphics{uvw.pstex}%
\end{picture}%
\setlength{\unitlength}{3947sp}%
\begingroup\makeatletter\ifx\SetFigFont\undefined%
\gdef\SetFigFont#1#2#3#4#5{%
  \reset@font\fontsize{#1}{#2pt}%
  \fontfamily{#3}\fontseries{#4}\fontshape{#5}%
  \selectfont}%
\fi\endgroup%
\begin{picture}(777,791)(1036,-470)
\put(1051, 14){\makebox(0,0)[lb]{\smash{{\SetFigFont{12}{14.4}{\familydefault}{\mddefault}{\updefault}{$u$}%
}}}}
\put(1756, 14){\makebox(0,0)[lb]{\smash{{\SetFigFont{12}{14.4}{\familydefault}{\mddefault}{\updefault}{$w$}%
}}}}
\put(1336,-406){\makebox(0,0)[lb]{\smash{{\SetFigFont{12}{14.4}{\familydefault}{\mddefault}{\updefault}{$v$}%
}}}}
\end{picture}%

\end{center}


Consider the \emph{lattice points} $iu+jv$ with integer $i,j$. A
segment joining two nearest lattice points will be called a \emph{small
edge}. We are interested in positive measures $m$ which are supported
by a union of small edges, whose restriction to each small edge is
a multiple of linear measure, and which satisfy the balance condition
(called zero tension in \cite{KTW})\begin{equation}
m(AB)-m(AB')=m(AC)-m(AC')=m(AD)-m(AD')\label{eq:balance}\end{equation}
whenever $A$ is a lattice point and the lattice points $B,C',D,B',C,D'$
are in cyclic order around $A$.

\begin{center} \begin{picture}(0,0)%
\includegraphics{BCD-cyclic.pstex}%
\end{picture}%
\setlength{\unitlength}{3947sp}%
\begingroup\makeatletter\ifx\SetFigFont\undefined%
\gdef\SetFigFont#1#2#3#4#5{%
  \reset@font\fontsize{#1}{#2pt}%
  \fontfamily{#3}\fontseries{#4}\fontshape{#5}%
  \selectfont}%
\fi\endgroup%
\begin{picture}(1380,1261)(886,-1775)
\put(1801,-661){\makebox(0,0)[lb]{\smash{{\SetFigFont{10}{10}{\familydefault}{\mddefault}{\updefault}{$B'$}%
}}}}
\put(2251,-1200){\makebox(0,0)[lb]{\smash{{\SetFigFont{10}{10}{\familydefault}{\mddefault}{\updefault}{$C$}%
}}}}
\put(1201,-1711){\makebox(0,0)[lb]{\smash{{\SetFigFont{10}{10}{\familydefault}{\mddefault}{\updefault}{$B$}%
}}}}
\put(901,-1200){\makebox(0,0)[lb]{\smash{{\SetFigFont{10}{10}{\familydefault}{\mddefault}{\updefault}{$C'$}%
}}}}
\put(1201,-661){\makebox(0,0)[lb]{\smash{{\SetFigFont{10}{10}{\familydefault}{\mddefault}{\updefault}{$D$}%
}}}}
\put(1951,-1711){\makebox(0,0)[lb]{\smash{{\SetFigFont{10}{10}{\familydefault}{\mddefault}{\updefault}{$D'$}%
}}}}
\put(1701,-1050){\makebox(0,0)[lb]{\smash{{\SetFigFont{10}{10}{\familydefault}{\mddefault}{\updefault}{$A$}%
}}}}
\end{picture}%

\end{center}If $e$ is a small edge, the value $m(e)$ is equal to the density
of $m$ relative to linear measure on that edge.

Fix now an integer $r\ge1$, and denote by $\triangle_{r}$ the (closed)
triangle with vertices $0,ru,$ and $ru+rv=-rw$. We will use the
notation $A_{j}=ju,B_{j}=ru+jv$, and $C_{j}=(r-j)w$ for the lattice
points on the boundary of $\triangle_{r}$. We also set \[
X_{j}=A_{j}+w,Y_{j}=B_{j}+u,Z_{j}=C_{j}+v\]
for $j=0,1,2,\dots,r+1$. The following picture represents $\triangle_{5}$
and the points just defined; the labels are placed on the left.\begin{center}

\begin{picture}(0,0)%
\includegraphics{basic-triangle.pstex}%
\end{picture}%
\setlength{\unitlength}{3947sp}%
\begingroup\makeatletter\ifx\SetFigFont\undefined%
\gdef\SetFigFont#1#2#3#4#5{%
  \reset@font\fontsize{#1}{#2pt}%
  \fontfamily{#3}\fontseries{#4}\fontshape{#5}%
  \selectfont}%
\fi\endgroup%
\begin{picture}(2303,2328)(2536,-2305)
\put(4251,-1336){\makebox(0,0)[lb]{\smash{{\SetFigFont{10}{10}{\familydefault}{\mddefault}{\updefault}{$Z_2$}%
}}}}
\put(4401,-1636){\makebox(0,0)[lb]{\smash{{\SetFigFont{10}{10}{\familydefault}{\mddefault}{\updefault}{$Z_1$}%
}}}}
\put(4551,-1936){\makebox(0,0)[lb]{\smash{{\SetFigFont{10}{10}{\familydefault}{\mddefault}{\updefault}{$Z_0$}%
}}}}
\put(4251,-1986){\makebox(0,0)[lb]{\smash{{\SetFigFont{10}{10}{\familydefault}{\mddefault}{\updefault}{$C_0$}%
}}}}
\put(3051,-1986){\makebox(0,0)[lb]{\smash{{\SetFigFont{10}{10}{\familydefault}{\mddefault}{\updefault}{$B_1$}%
}}}}
\put(2751,-1986){\makebox(0,0)[lb]{\smash{{\SetFigFont{10}{10}{\familydefault}{\mddefault}{\updefault}{$B_0$}%
}}}}
\put(3201,-2236){\makebox(0,0)[lb]{\smash{{\SetFigFont{10}{10}{\familydefault}{\mddefault}{\updefault}{$Y_2$}%
}}}}
\put(2901,-2236){\makebox(0,0)[lb]{\smash{{\SetFigFont{10}{10}{\familydefault}{\mddefault}{\updefault}{$Y_1$}%
}}}}
\put(2601,-2236){\makebox(0,0)[lb]{\smash{{\SetFigFont{10}{10}{\familydefault}{\mddefault}{\updefault}{$Y_0$}%
}}}}
\put(3501,-436){\makebox(0,0)[lb]{\smash{{\SetFigFont{10}{10}{\familydefault}{\mddefault}{\updefault}{$A_0$}%
}}}}
\put(3351,-136){\makebox(0,0)[lb]{\smash{{\SetFigFont{10}{10}{\familydefault}{\mddefault}{\updefault}{$X_0$}%
}}}}
\put(3201,-436){\makebox(0,0)[lb]{\smash{{\SetFigFont{10}{10}{\familydefault}{\mddefault}{\updefault}{$X_1$}%
}}}}
\put(4101,-1636){\makebox(0,0)[lb]{\smash{{\SetFigFont{10}{10}{\familydefault}{\mddefault}{\updefault}{$C_1$}%
}}}}
\put(3351,-736){\makebox(0,0)[lb]{\smash{{\SetFigFont{10}{10}{\familydefault}{\mddefault}{\updefault}{$A_1$}%
}}}}
\put(3051,-736){\makebox(0,0)[lb]{\smash{{\SetFigFont{10}{10}{\familydefault}{\mddefault}{\updefault}{$X_2$}%
}}}}
\end{picture}%

\end{center}Given a measure $m$, a \emph{branch point} is a lattice point incident
to at least three edges in the support of $m$. We will only consider
measures with at least one branch point. This excludes measures whose
support consists of one or more parallel lines. We will denote by
$\mathcal{M}_{r}$ the collection of all measures $m$ satisfying
the balance condition above, whose branch points are contained in
$\triangle_{r}$, and such that\[
m(A_{j}X{}_{j+1})=m(B_{j}Y{}_{j+1})=m(C_{j}Z{}_{j+1})=0,\quad j=0,1,\dots,r.\]
 Analogously $\mathcal{M}_{r}^{*}$ consists of measures $m$ whose
branch points are contained in $\mathcal{M}_{r}$, and such that\[
m(A_{j}X{}_{j})=m(B_{j}Y{}_{j})=m(C_{j}Z{}_{j})=0,\quad j=0,1,\dots,r.\]
Clearly, $\mathcal{M}_{r}^{*}$ can be obtained from $\mathcal{M}_{r}$
by reflection relative to one of the angle bisectors of $\triangle_{r}$.

Given a measure $m\in\mathcal{M}_{r}$, we define its \emph{weight}
$\omega(m)\in\mathbb{R}_{+}$ to be\[
\omega(m)=\sum_{j=0}^{r}m(A_{j}X_{j})=\sum_{j=0}^{r}m(B_{j}Y_{j})=\sum_{j=0}^{r}m(C_{j}Z_{j})\]
and its \emph{boundary} $\partial m=(\alpha,\beta,\gamma)\in(\mathbb{R}^{r})^{3}$,
where\[
\alpha_{\ell}=\sum_{j=0}^{\ell-1}m(A_{j}X_{j}),\beta_{\ell}=\sum_{j=0}^{\ell-1}m(B_{j}Y_{j}),\gamma_{\ell}=\sum_{j=0}^{\ell-1}m(C_{j}Z_{j}),\quad\ell=1,2,\dots,r.\]
The equality of the three sums giving $\omega(m)$ is an easy consequence
of the balance condition. 

The results of \cite{KT,KTW} imply that the sets $I,J,K\subset\{1,2,\dots,n\}$
of cardinality $r$ satisfy $c_{IJK}>0$ if and only if there exists
a measure $m\in\mathcal{M}_{r}$ with weight $\omega(m)=n-r$, and
with boundary $\partial m=(\alpha,\beta,\gamma)$ such that\[
\alpha_{\ell}=I(\ell)-\ell,\beta_{\ell}=J(\ell)-\ell,\gamma_{\ell}=K(\ell)-\ell,\quad\ell=1,2,\dots,r.\]
The number $c_{IJK}$ is equal to the number of measures in $\mathcal{M}_{r}$
satisfying these conditions, and with integer densities on all edges.
Moreover, as shown in \cite{KTW}, if $c_{IJK}=1$, there is only
one measure $m$ satisfying these conditions, and its densities must
naturally be integers. In general, we will say that a measure $m\in\mathcal{M}_{r}$
is \emph{rigid} if it is entirely determined by its weight and boundary.

We will also use the version of these results in terms of $\mathcal{M}_{r}^{*}$,
so we define for $m\in\mathcal{M}_{r}^{*}$ the weight \[
\omega(m)=\sum_{j=0}^{r}m(A_{j}X{}_{j+1})=\sum_{j=0}^{r}m(B_{j}Y{}_{j+1})=\sum_{j=0}^{r}m(C_{j}Z{}_{j+1})\]
and boundary $\partial m=(\alpha,\beta,\gamma)$, where\[
\alpha_{\ell}=\sum_{j=r+1-\ell}^{r}m(A_{j}X{}_{j+1}),\beta_{\ell}=\sum_{j=r+1-\ell}^{r}m(B_{j}Y{}_{j+1}),\gamma_{\ell}=\sum_{j=r+1-\ell}^{r}m(C_{j}Z{}_{j+1})\]
for $\ell=1,2,\dots,r$.

Measures in $\mathcal{M}_{r}$ or $\mathcal{M}_{r}^{*}$ are entirely
determined by their restrictions to $\triangle_{r}$ and, when the
corners of $\triangle_{r}$ are not branch points, even by their restrictions
to the interior of $\triangle_{r}$. Indeed, the lack of branch points
outside $\triangle_{r}$ implies that the densities are constant on
the half-lines starting with $A_{j}X_{j},B_{j}Y_{j}$ and $C_{j}Z_{j}$.
Note that a restriction $m|\triangle_{r}$ with $m\in\mathcal{M}_{r}$
is not generally of the form $m'|\triangle_{r}$ for some $m'\in\mathcal{M}_{r}^{*}$.
The first picture below represents $\triangle_{r}$ (dotted lines),
and the support (solid lines) of a measure in $\triangle_{r}$. The
second one represents the support of a measure in $\mathcal{M}_{r}^{*}$.
\begin{center}

\includegraphics{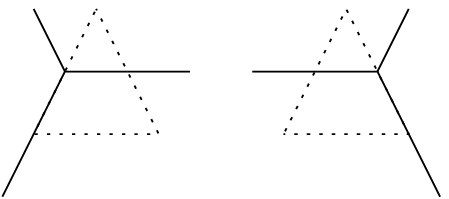}

\end{center}

To conclude this section, we establish a connection between measures
in $\mathcal{M}_{r}$ and the honeycombs of \cite{KT}. A \emph{honeycomb}
is a function $h$ defined on the set of small edges contained in
$\triangle_{r}$ satisfying the following two properties:

\begin{enumerate}

\item[(i)]If $ABC$ is a small triangle contained in $\triangle_{r}$, we have
$h(AB)+h(AC)+h(BC)=0$.\item[(ii)]If $A,B,C,D$ are lattice vertices
in $\triangle_{r}$ such that $B=A+u,$ $C=A-v$, $D=A+w$ (or $B=A+v$,
$C=A-w$, $D=A+u$, or $B=A+w$, $C=A-u$, $D=A+v$), then\[
h(AB)-h(CD)=h(BC)-h(AD)\ge0.\]
\end{enumerate}

The reason for the term honeycomb is not visible in our definition.
One can associate to each small triangle $ABC\subset\triangle_{r}$
the point $(h(AB),h(BC),h(AC))$ in the plane $\{(x,y,z)\in\mathbb{R}^{3}:x+y+z=0\}$.
These points form the vertices of a graph which looks like a honeycomb
if it is not too degenerate (cf. \cite{KT}).

The following result will be required for our discussion of the Horn
inequalities in Section 4.

\begin{lem}
\label{lem:Honeycomb-from-m}Let $m\in\mathcal{M}_{r}$ be a measure
with weight $\omega$ and $\partial m=(\alpha,\beta,\gamma)$. There
exists a honeycomb $h$ with the following properties.
\begin{enumerate}
\item $h(A_{\ell-1}A_{\ell})=\alpha_{\ell}-2\omega/3$, $h(B_{\ell-1}B_{\ell})=\beta_{\ell}-2\omega/3,$
$h(C_{\ell-1}C_{\ell})=\gamma_{\ell}-2\omega/3$ for $\ell=1,2,\dots,r$.
\item If $B=A+u,C=A-v,D=A+w$ $($or $B=A+v,C=A-w,D=A+u$, or $B=A+w$,
$C=A-u$, $D=A+v)$ then $h(AB)-h(CD)=m(AC)$.
\end{enumerate}
\end{lem}
\begin{proof}
This is routine. Condition $(2)$ allows us to calculate all the values
of $h$ starting from the boundary of $\triangle_{r}$. To verify
(ii) one must use the balance condition for measures in $\mathcal{M}_{r}$.
\end{proof}

\section{Inflation, Duality, and Rigidity}

Let $\mathcal{A}$ be a finite factor with trace normalized so that
$\tau(1)=n$, and let $\mathcal{E}=\{ E_{\ell}:\ell=0,1,\dots,n\}$
be a flag so that $\tau(E_{\ell})=\ell$ for all $\ell$. Fix also
a set $I\subset\{1,2,\dots,n\}$ of cardinality $r$ and a projection
$P\in\mathcal{A}$ with $\tau(P)=r$. We have $P\in S(\mathcal{E},I)$,
i.e. $\tau(P\wedge E_{I(\ell)})\ge\ell$ for $\ell=1,2,\dots,r$,
if and only if $\tau(P\wedge E_{\ell})\ge\varphi_{I}(\ell)$ for $\ell=0,1,\dots,n$,
where \[
\varphi_{I}(\ell)=p\text{ for }I(p)\le\ell<I(p+1)\]
for $p=0,1,\dots,r,$, and $I(0)=1,I(r+1)=n+1$. With the notation
$P^{\perp}=1-P$, these conditions imply \begin{eqnarray*}
\tau(P^{\perp}\wedge E_{\ell}^{\perp}) & = & n-\tau(P\vee E_{\ell})\\
 & = & n-\tau(P)-\tau(E_{\ell})+\tau(P\wedge E_{\ell})\\
 & \ge & n-r-\ell+\varphi_{I}(\ell).\end{eqnarray*}
This implies that $P\in S(\mathcal{E},I)$ if and only if $P^{\perp}\in S(\mathcal{E}^{\perp},I^{*})$,
where $\mathcal{E}^{\perp}=\{ E_{n-\ell}^{\perp}:\ell=0,1,\dots,n$\}
, and $I^{*}=\{ n+1-i:i\notin I\}$. In general, we will have $c_{I^{*}J^{*}K^{*}}=c_{IJK}$,
and this equality is realized by a duality considered in \cite{KTW}.
More precisely, assume that $m\in\mathcal{M}_{r}$. We define the
\emph{inflation} of $m$ as follows. Cut $\triangle_{r}$ along the
edges in the support of $m$ to obtain a collection of (white) puzzle
pieces, and translate these pieces away from each other in the following
way: the parallelogram formed by the two translates of a side $AB$
of a white puzzle piece has two sides of length equal to the density
of $m$ on $AB$ and $60^{\circ}$ clockwise from $AB$. The original
puzzle pieces and these parallelograms fit together, and leave a space
corresponding to each branch point in the support of $m$. Here is
an illustration of the process; the thinner lines in the support of
the measure have density one, and the thicker ones density 2. The
original pieces of the triangle $\triangle_{r}$ are white, the added
parallelogram pieces are dark gray, and the branch points become light
gray pieces. Each light gray piece has as many sides as there are
branches at the original branch point (counting the branches outside
$\triangle_{r}$, which are not represented in this figure, though
their number and densities are dictated by the balance condition,
and the fact that $m$ belongs to $\mathcal{M}_{r}$).

\begin{center}

\includegraphics[scale=0.7]{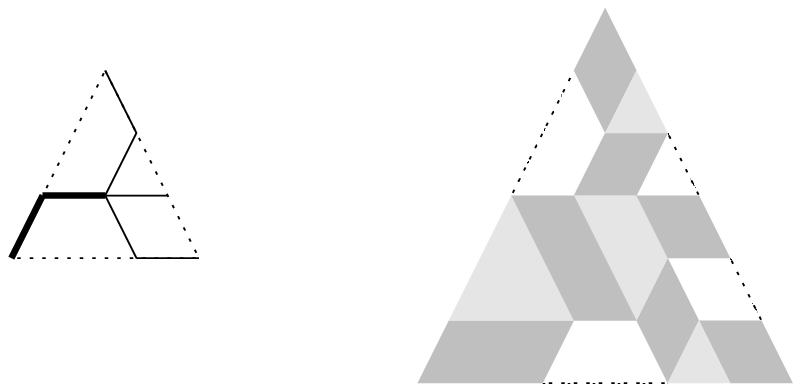}

\end{center}The original triangle $\triangle_{r}$ has been inflated to a triangle
of size $r+\omega(m)$, and the decomposition of this triangle into
white, gray, and light gray pieces is known as the \emph{puzzle} associated
to \emph{}$m$. Each gray parallelogram has two \emph{light gray sides},
i.e. sides bordering a light gray piece, and two \emph{white sides.}
The length of the light gray side equals the density of the white
sides in the original support of $m$ in $\triangle_{r}$. This process
can be applied to the entire support of $m$, but we are only interested
in its effect on $\triangle_{r}$. (The white regions in the puzzle
are called `zero regions', and the light gray ones `one regions' in
\cite{KTW}. We use in our drawings a color scheme different from
the one used in \cite{KTW}.)

We can now apply a dual deflation, or {*}\emph{deflation}, to the
puzzle of $m$ as follows: discard all the white pieces, and shrink
the gray parallelograms by reducing their white sides to points. The
segments obtained this way are assigned densities equal to the lengths
of the white sides of the corresponding parallelograms. In the picture
below, the shrunken parallelograms are represented as solid lines.
\begin{center}

\includegraphics[scale=0.7]{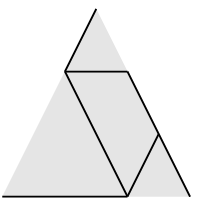}

\end{center}The result of this deflation is a triangle with sides $\omega(m)$,
endowed with a measure supported by the solid lines which will be
denoted $m^{*}$. The support of $m^{*}$ can be obtained directly
from the support of $m$ as follows: take every edge of a white puzzle
piece, rotate it $60^{\circ}$ clockwise, and change its length to
the density of $m$ on the original edge. The new segments must now
be translated so that the edges originating from the sides of a white
puzzle piece become concurrent. Thus, the dual picture depends primarily
on the combinatorial structure of the support of $m$. More precisely,
let us say that the measures $m\in\mathcal{M}_{r}$ and $m'\in\mathcal{M}_{r'}$
are \emph{homologous} if there is a bijection between the edges determined
by the support of $m$ and the edges determined by the support of
$m'$ such that corresponding edges are parallel, and concurrent edges
correspond with concurrent edges (the concurrence point being precisely
the one dictated by the correspondence of the edges). Then $m$ and
$m'$ are homologous if and only if $m^{*}$ and $m^{\prime*}$ are
homologous. For instance, measures in $\mathcal{M}_{r}$ that have
the same support are homologous. 

Assume now that the measure $m\in\mathcal{M}_{r}$ has integer densities,
and $I,J,K$ are the corresponding sets in $\{1,2,\dots,n=r+\omega(m)\}$.
Then the triangle obtained by inflating $m$ can be identified with
$\triangle_{n}$. Under this identification the small edges $A_{i-1}A_{i}$
are either white (if they border a white piece, or they belong to
a white edge of a gray parallelogram) or light gray. It is easy to
see that the white small edges are precisely $A_{i_{\ell}-1}A_{i_{\ell}}$
for $\ell=1,2,\dots,r$, and therefore the light gray edges correspond
with the complement of $I$. Furthermore, the light gray triangle
obtained by {*}deflation can be identified with $\triangle_{n-r}$,
and $m^{*}\in\mathcal{M}_{n-r}^{*}$ is a measure satisfying $\omega(m^{*})=r$.
This measure determines subsets of $\{1,2,\dots,n\}$ which are precisely
$I^{*},J^{*},K^{*}$. This observation gives a bijective proof of
the equality $c_{IJK}=c_{I^{*}J^{*}K^{*}}$.

The passage from $m$ to $m^{*}$ can be reversed by applying {*}inflation
to $m^{*}$, and then applying deflation to the resulting puzzle. 

Another important application of the inflation process is a characterization
of rigidity. Orient the edges of the gray parallelograms in a puzzle
so that they point away from the acute angles. Some of the border
edges do not have a neighboring gray parallelogram and will not be
oriented.\begin{center}

\includegraphics[scale=0.7]{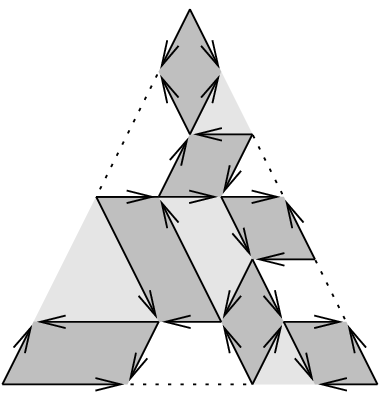}

\end{center}It was shown in \cite{KTW} that a measure $m$ is rigid if and only
if the associated directed graph contains no gentle loops, i.e., loops
which never turn more than $60^{\circ}$. Note that the number and
relative position of the puzzle pieces depends only on the support
of the measure $m$. The following result follows immediately.

\begin{prop}
Let $m_{1},m_{2}\in\mathcal{M}_{r}$ be such that the support of $m_{1}$
is contained in the support of $m_{2}$. If $m_{2}$ is rigid then
$m_{1}$ is rigid as well.
\end{prop}
It is also easy to see that the support of a rigid measure does not
contain six edges which meet at the same point. Indeed, the inflation
reveals immediately a gentle loop.\begin{center}

\includegraphics[scale=0.7]{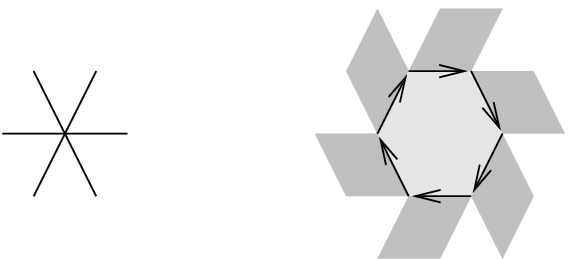}

\end{center}

We will need to characterize rigidity in terms of the support of the
original measure. Let $A_{1}A_{2}\cdots A_{k}A_{1}$ be a loop consisting
of small edges $A_{j}A_{j+1}$ contained in the support of a measure
$m\in\mathcal{M}_{r}$. We will say that this loop is \emph{evil}
if each three consecutive points $A_{j-1}A_{j}A_{j+1}=ABC$ form an
\emph{evil turn,} i.e. one of the following situations occurs:

\begin{enumerate}
\item $C=A$, and the small edges $BX,BY,BZ$ which are $120^{\circ},180^{\circ}$,
and $240^{\circ}$ clockwise from $AB$ are in the support of $m$.
\item $BC$ is $120^{\circ}$ clockwise from $AB$.
\item $C\ne A$ and $A,B,C$ are collinear.
\item $BC$ is $120^{\circ}$ counterclockwise from $AB$ and the edge $BX$
which is $120^{\circ}$ clockwise from $AB$ is in the support of
$m$.
\item $BC$ is $60^{\circ}$ counterclockwise from $AB$ and the edges $BX,BY$
which are $120^{\circ}$ and $180^{\circ}$ clockwise from $AB$ are
in the support of $m$.
\end{enumerate}
\begin{prop}
\label{prop:absence-of-evil}A measure $m\in\mathcal{M}_{r}$ is rigid
if and only if its support contains none of the following configurations:
\begin{enumerate}
\item Six edges meeting in one lattice point; 
\item An evil loop.
\end{enumerate}
\end{prop}
\begin{proof}
Assume first that $m$ is not rigid, and consider a gentle loop of
minimal length in its puzzle. Use $a_{i}$ to denote white parallelogram
sides in this loop, and $b_{i}$ light gray parallelogram sides. The
sides $a_{i},b_{i}$ may consist of several small edges. The gentle
loop is of one of the following three forms:
\begin{enumerate}
\item $a_{1}a_{2}a_{3}a_{4}a_{5}a_{6}$, 
\item $b_{1}b_{2}b_{3}b_{4}b_{5}b_{6}$,
\item $a_{1}a_{2}\dots a_{i_{1}}b_{1}b_{2}\cdots b_{j_{1}}a_{i_{1}+1}\cdots a_{i_{2}}b_{j_{1}+1}\cdots b_{j_{2}}\cdots a_{i_{p-1}+1}\cdots a_{i_{p}}b_{j_{p-1}+1}\cdots b_{j_{p}},$
with at most five consecutive $a$ or $b$ symbols.
\end{enumerate}
In case (1), the loop runs counterclockwise around a white piece,
and it deflates to a translation of itself which is obviously evil.
In case $(2)$, the loop deflates to a single point where six edges
in the support of $m$ meet. In case $(3)$, the loop deflates to
$a'_{1}a'_{2}\cdots a'_{i_{p}},$ where each $a'_{j}$ is a translate
of $a_{j}$. The turns in this loop are obtained by deflating a path
of the form $a_{1}b_{1}\cdots b_{j}a_{2}$ with $0\le j\le5$. The
edges $b_{1}\cdots b_{j}$ run clockwise around a light gray piece
which must have at most five edges because the gentle loop was taken
to have minimal length. When $j=0$, the edges $a_{1}$ and $a_{2}$
border the same white puzzle piece, and it is obvious that $a'_{1}a'_{2}$
is an evil turn. The remaining cases will be enumerated according
to the number of edges in the light gray piece next to the edges $b_{j}$.When
this piece is a triangle, we can only have $j=1$, and the situation
is illustrated below. The dashed line in the deflation indicates a
portion of the support of $m$.\begin{center}

\includegraphics[scale=0.7]{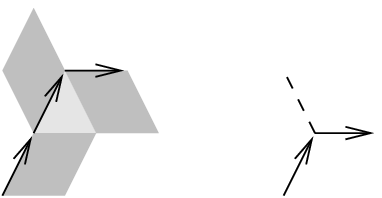}

\end{center} When the light gray piece is a parallelogram, we have $j=1$ or $j=2$.
The three possible deflations are as follows.\begin{center}

\includegraphics[scale=0.7]{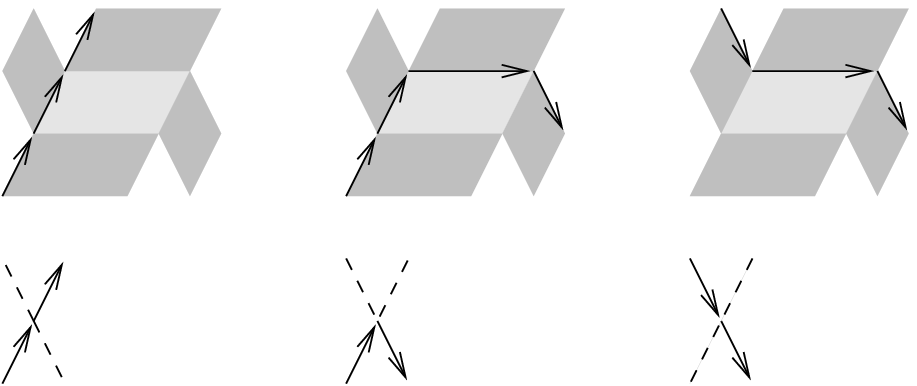}

\end{center}Next, the light gray piece may be a trapezoid, and $1\le j\le3$.
For $j=1$ we have these four possibilities:\begin{center}

\includegraphics[scale=0.7]{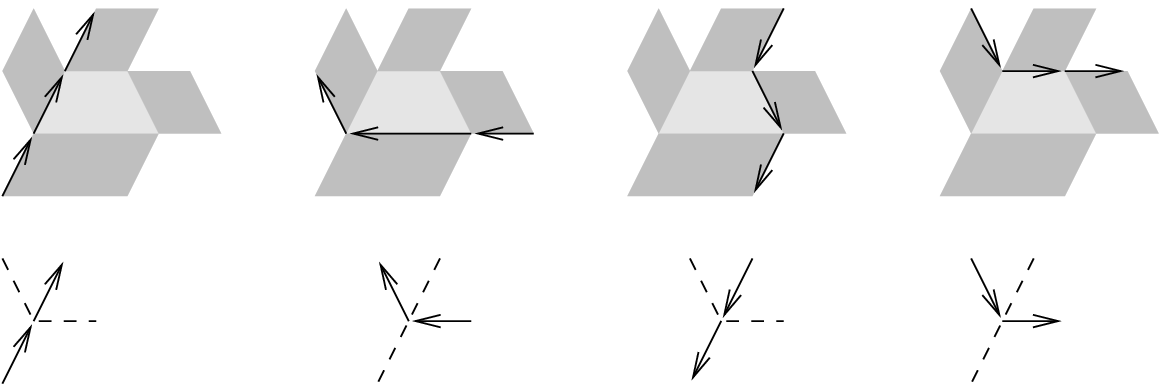}

\end{center}For $j=2,3$ there are three more possibilities.\begin{center}

\includegraphics[scale=0.7]{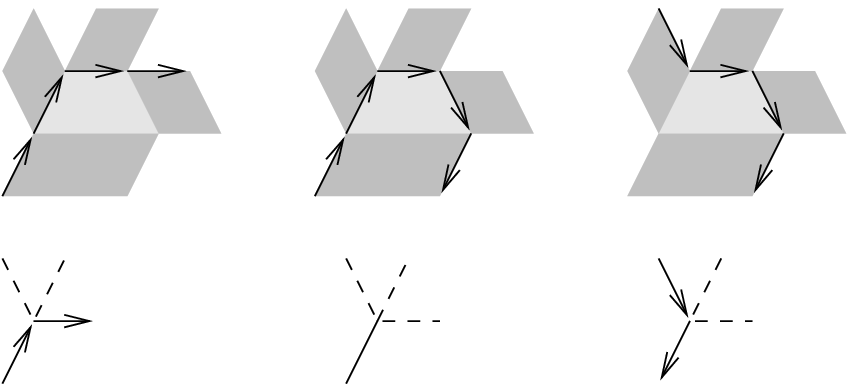}

\end{center}Finally, if the light gray piece is a pentagon, there are five situations
when $j=1$,\begin{center}

\includegraphics[scale=0.7]{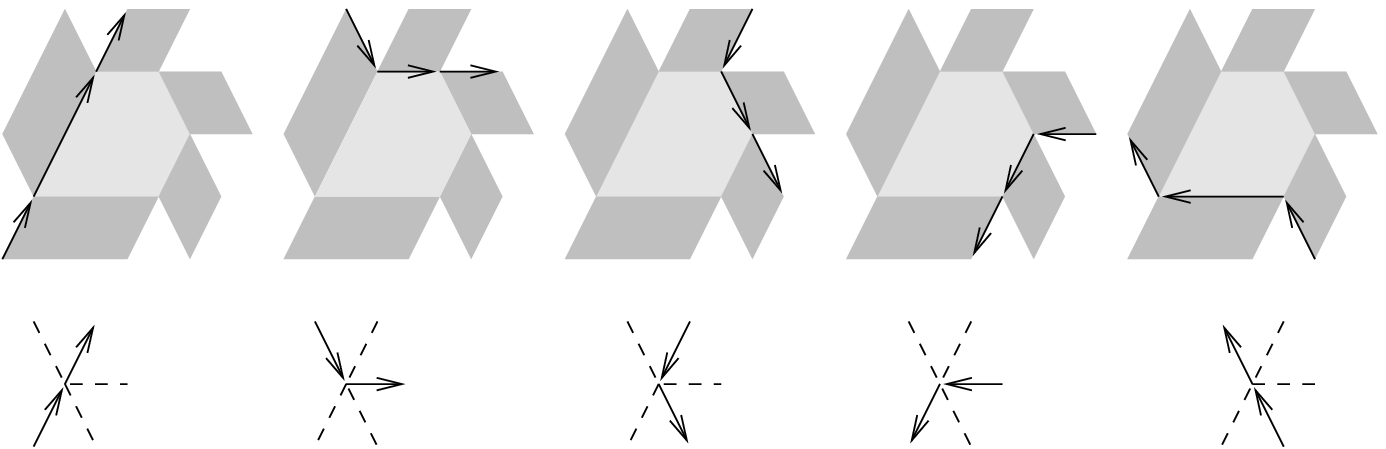}

\end{center}four situations when $j=2$,\begin{center}

\includegraphics[scale=0.7]{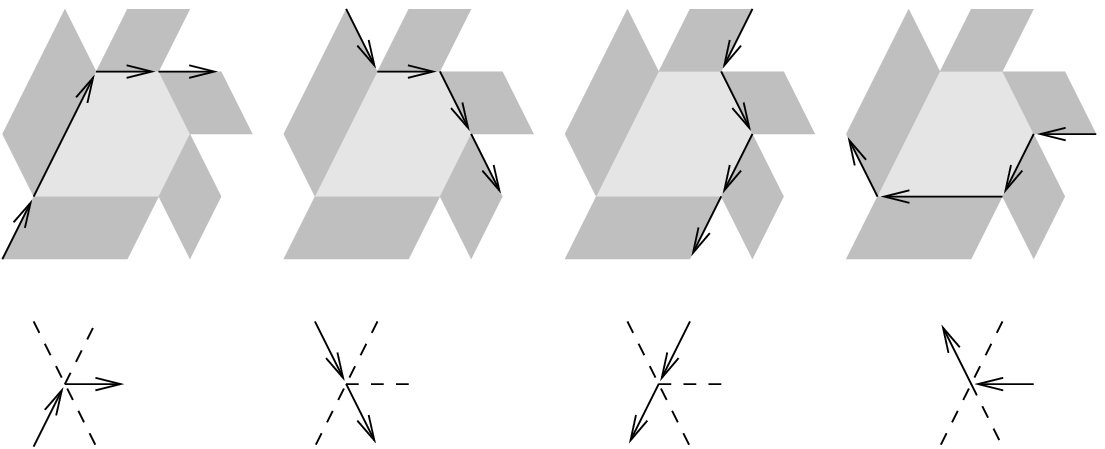}

\end{center}three situations when $j=3$,\begin{center}

\includegraphics[scale=0.7]{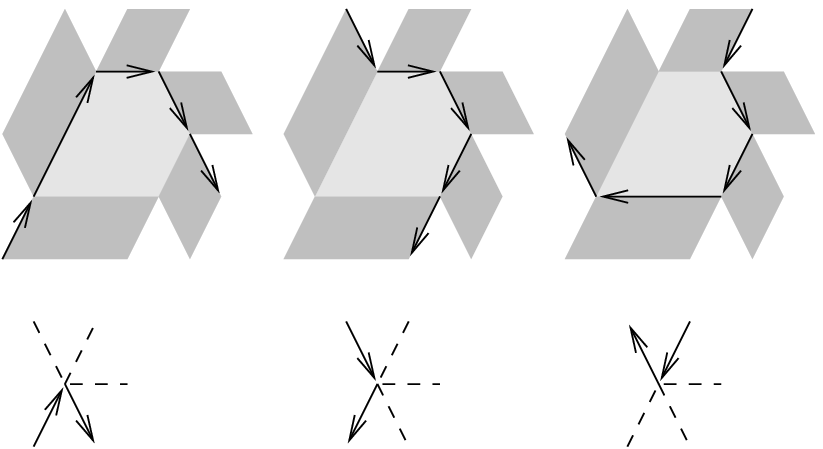}

\end{center}and three more when $j=4,5$.\begin{center}

\includegraphics[scale=0.7]{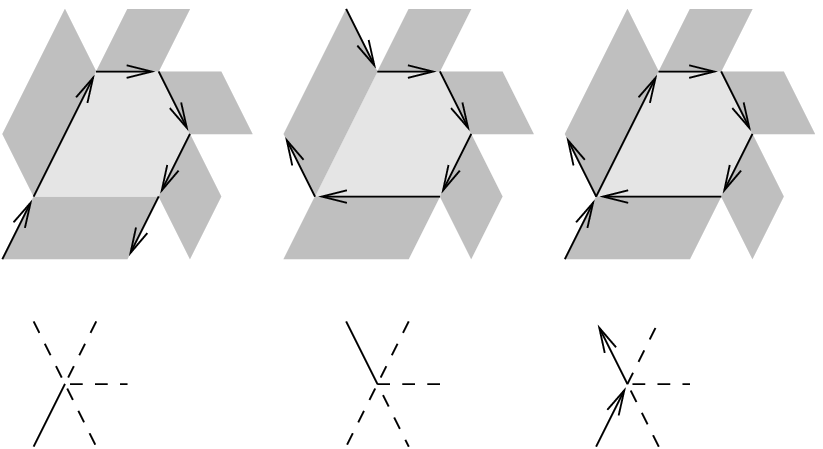}

\end{center}Of course, the last case does not occur in a minimal gentle loop.

Thus, in all situations, the deflated turns are evil, and therefore
the support of $m$ contains an evil loop. Conversely, if the support
of $m$ contains an evil loop, the above figures show that one can
obtain a gentle loop in the puzzle of $m$.
\end{proof}
The following figure represents the suport of a rigid measure in $\mathcal{M}_{8}$,
along with a loop which may seem evil but is not evil when traversed
in either direction.\begin{center}

\includegraphics[scale=0.7]{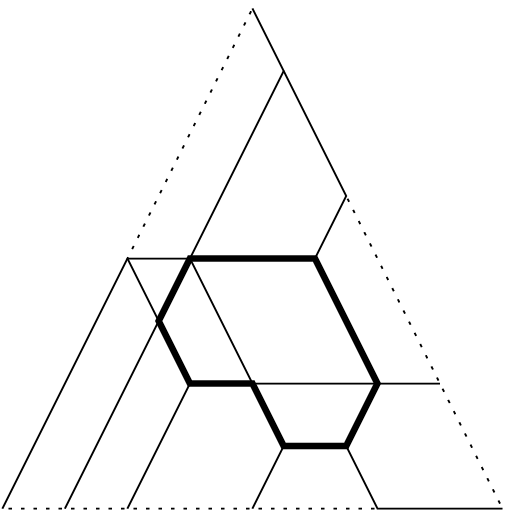}

\end{center}

\section{Extremal Measures and Skeletons}

For fixed $r\ge1$, the collection $\mathcal{M}_{r}$ is a convex
polyhedral cone. Recall that a measure $m\in\mathcal{M}_{r}$ is \emph{extremal}
(or belongs to an extreme ray) if any measure $m'\le m$ is a positive
multiple of $m$. The support of an extremal measure will be called
a \emph{skeleton}. Clearly, an extremal measure is entirely determined
by its value on any small edge contained in its skeleton. Checking
extremality is easily done by using the balance condition (\ref{eq:balance})
at all the branch points of the support to see how the density propagates
from one edge to the rest of the support.

In the following figure of a skeleton, the thicker edges must be assigned
twice the density of the other  edges. (This skeleton contains an
evil loop, hence the measures it supports are not rigid.)

\begin{center} \includegraphics[scale=0.7]{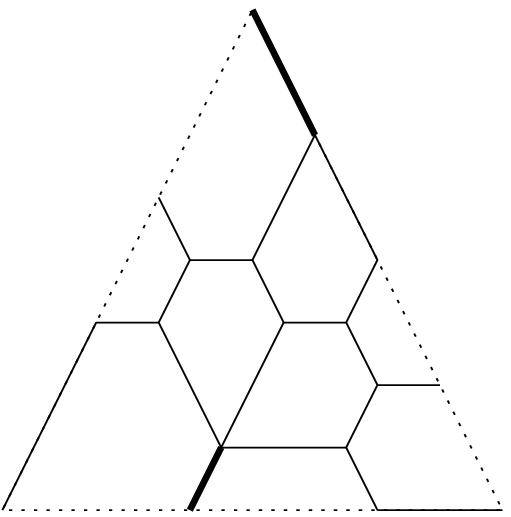}\end{center} 

It is not always obvious when a collection of edges supports a nonzero
measure in $\mathcal{M}_{r}$. The reader may find it amusing to verify
that the following figure represents sets which do not support any
nonzero measure.

\begin{center} \includegraphics[scale=0.5]{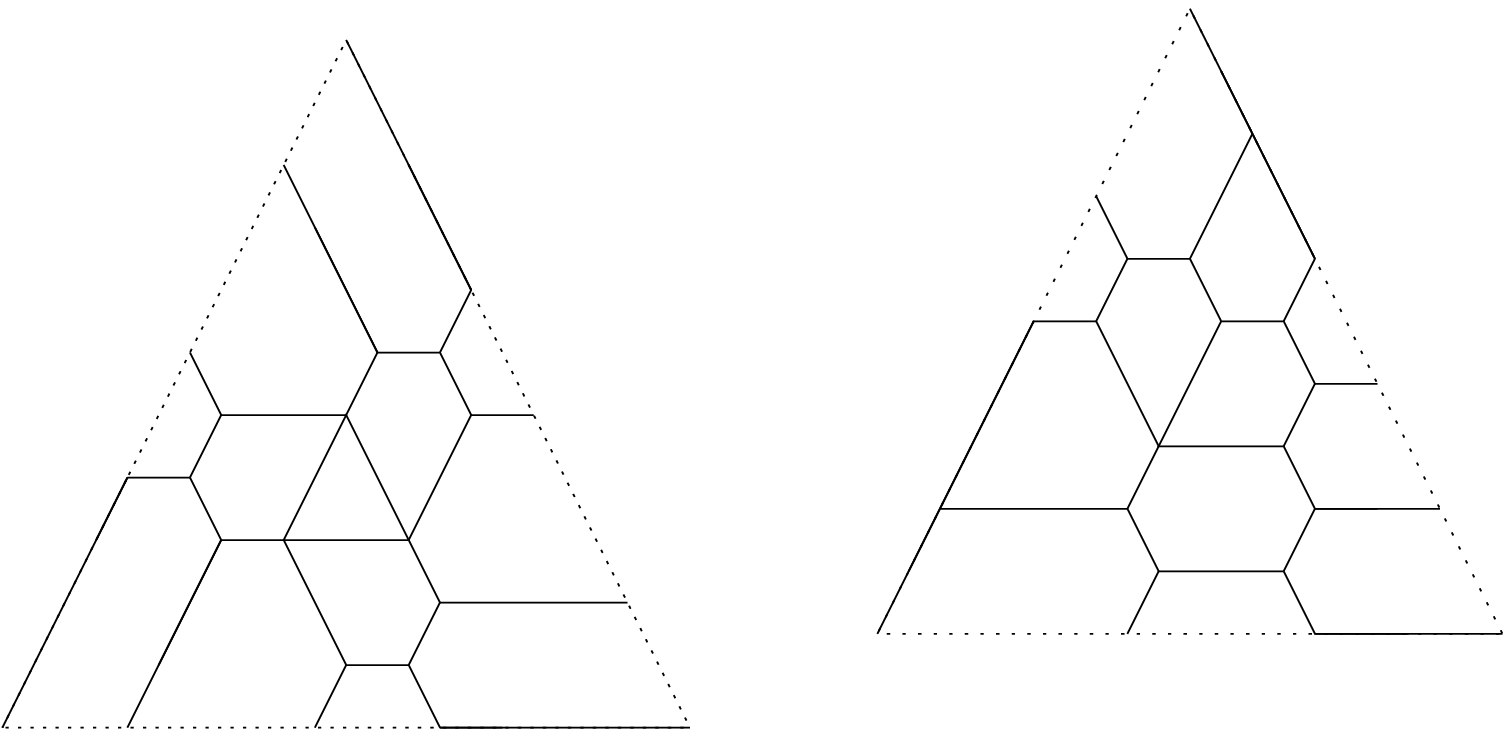} \end{center}

We will be mostly interested in the supports of rigid extremal measures,
which we will call rigid skeletons. When $r=1$, there are only the
three possible skeletons, all of them rigid, pictured below. \begin{center} \includegraphics[scale=0.7]{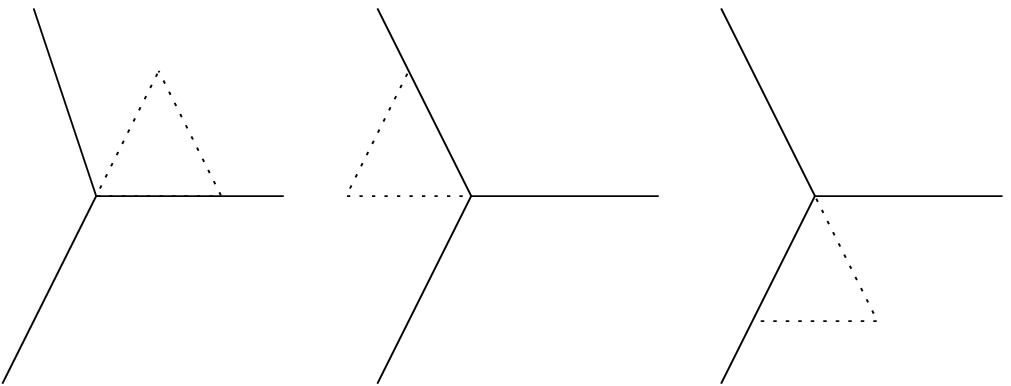} \end{center}The
following figure shows some rigid skeletons for $r=2,3,4,5$.

\begin{center} \includegraphics[scale=0.7]{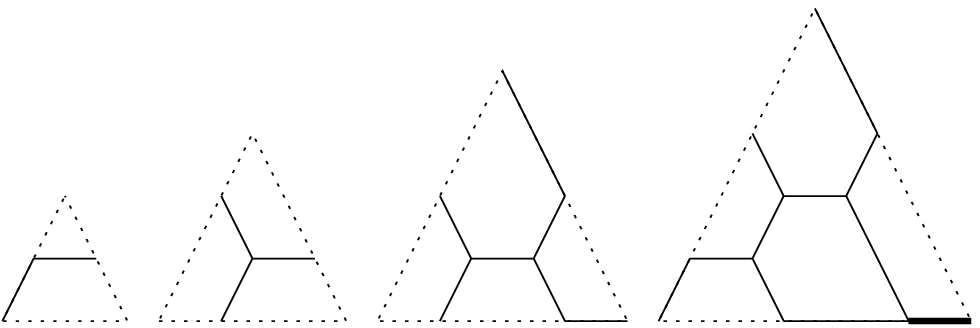} \end{center}

A greater variety of rigid skeletons is available for $r=6$. In addition
to larger versions (plus rotations and reflections) of the above skeletons,
we have the ones in the next figure.

\begin{center} \includegraphics[scale=0.7]{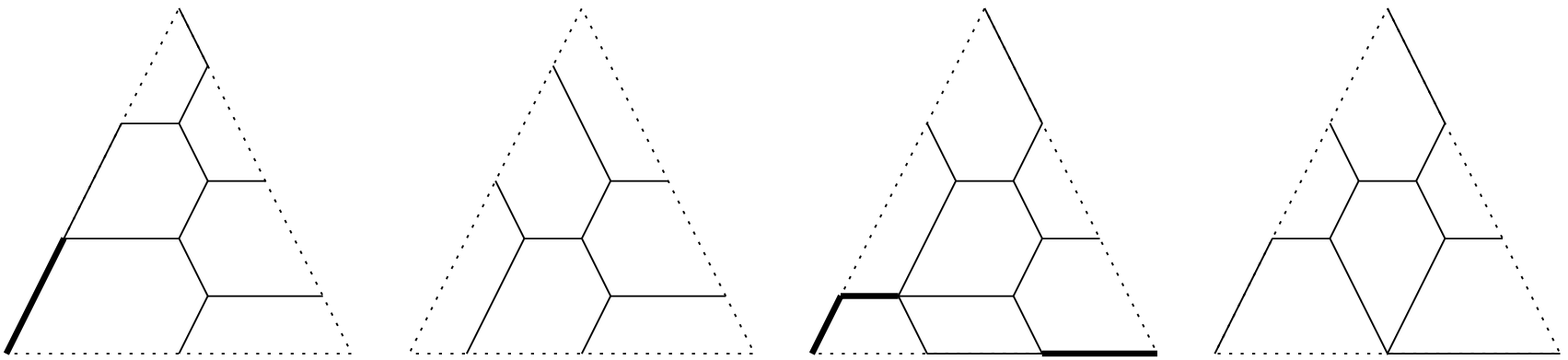} \end{center}

For larger $r$, rigid skeletons can be quite involved. We provide
just one more example for $r=8$. This skeleton has edges with densities
2 and 3 which we did not indicate.

\begin{center} \includegraphics[scale=0.7]{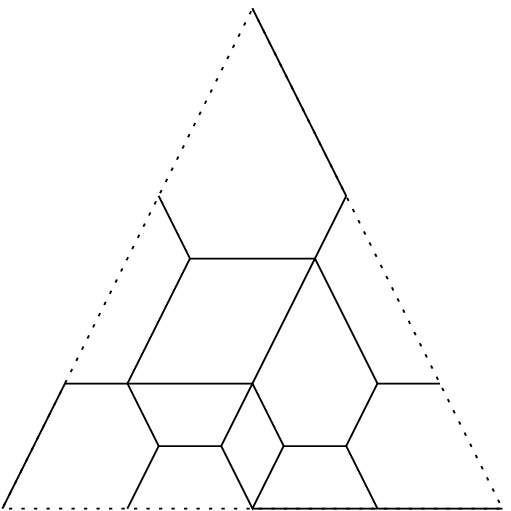} \end{center}

An important consequence of the characterization of rigid measures
in $\mathcal{M}_{r}$ is the fact that such measures can be written
uniquely as sums of extremal measures. Fix a rigid measure $m\in\mathcal{M}_{r}$,
and let $e=AB$ and $f=BC$ be two distinct small edges. We will write
$e\to_{m}f$, or simply $e\to f$ when $m$ is understood, if one
of these two situations arises:

\begin{enumerate}
\item $\angle ABC=120^{\circ}$ and the edge opposite $e$ at $B$ has $m$
measure equal to zero;
\item $e$ and $f$ are opposite, and one of the edges making an angle of
$60^{\circ}$ with $f$ has $m$ measure equal to zero.
\end{enumerate}
Note that in both cases we may also have $f\to e$. The significance
of this relation is that $e\to f$ implies that $m(e)\le m(f)$, with
strict inequality unless $f\to e$ as well. More precisely, if $e\to f$
and we do not have $f\to e$, there is at least one edge $g$ such
that $g\to f$ and the angle between $e$ and $g$ is $60^{\circ}$.
In this case we have\begin{equation}
m(f)=m(e)+m(g).\label{eq:descendant-of-two}\end{equation}
 Indeed, if $e$ and $f$ are opposite, the edge opposite $g$ must
have $m$ measure equal to zero.

A useful observation is that if $XY\to YZ$ but $YZ\not\to XY$, then
the edges $YA,YB,YC$ which are $120^{\circ},180^{\circ},240^{\circ}$
clockwise from $YZ$ must be in the support of $m$. In other words,
$ZYZ$ is an evil turn.

The following result is a simple consequence of the fact that the
measure $m$ exists.

\begin{lem}
Assume that a sequence of edges $e_{1},e_{2},\dots,e_{n}$ in the
support of $m$ is such that \[
e_{1}\to e_{2}\to e_{3}\to\cdots\to e_{n}\to e_{1}.\]
Then we also have\[
e_{n}\to e_{n-1}\to\cdots\to e_{1}\to e_{n}.\]

\end{lem}
\begin{proof}
Indeed, if one of the arrows cannot be reversed, then $m(e_{1})<m(e_{1})$. 
\end{proof}
We can therefore define a preorder relation on the set of small edges
as follows. Given two small edges  $e,f$, we write $e\Rightarrow f$
if either $e=f$, or there exist edges $e_{1},e_{2},\dots,e_{n}$
such that \[
e=e_{1}\to e_{2}\to\cdots\to e_{n}=f.\]
In this case, we will say that $f$ is a \emph{descendant} of $e$
and $e$ is an \emph{ancestor} of $f$. Two edges are \emph{equivalent},
$e\Leftrightarrow f$, if $e\Rightarrow f$ and $f\Rightarrow e$.
The relation of descendance becomes an order relation on the equivalence
classes of small edges. An edge $e$ will be called a \emph{root}
if $m(e)\ne0$ and $e$ belongs to a minimal class relative to descendance.
Clearly, every edge in the support of $m$ is a descendant of at least
one root.

If $m(e)\ne0$ and $e\Rightarrow f$ , then there exists a path $A_{0}A_{1}\cdots A_{k}$
in the support of $m$ such that $A_{j-1}A_{j}\to A_{j}A_{j+1}$ for
all $j$, $A_{0}A_{1}=e$, and $A_{k-1}A_{k}=f$. Paths of this form
will be referred to as \emph{descendance paths} from $e$ to $f$.
All the turns $A_{j-1}A_{j}A_{j+1}$ and $A_{j+1}A_{j}A_{j-1}$ in
a descendance path are evil.

\begin{lem}
Assume that $e$ and $f$ are in the support of a rigid measure $m\in\mathcal{M}_{r}$,
and $A_{0}A_{1}\cdots A_{k}$ and $B_{0}B_{1}\cdots B_{\ell}$ are
two descendance paths from $e$ to $f$. Then $A_{k-1}=B_{\ell-1}$
and $A_{k}=B_{\ell}$.
\end{lem}
\begin{proof}
Assume to the contrary that $A_{k-1}=B_{\ell}$. There are two cases
to consider, according to whether $A_{0}=B_{1}$ or $A_{0}=B_{0}$.
In the first case, the loop \[
A_{0}A_{1}\cdots A_{k-1}B_{\ell-1}B_{\ell-2}\cdots B_{1}\]
 is evil, contradicting the rigidity of $m$. In the second case,
there is a first index $p$ such that $A_{p+1}\not=B_{p+1}$. Then
the loop \[
A_{p}A_{p+1}\cdots A_{k-1}B_{\ell-1}B_{\ell-2}\cdots B_{p}\]
 is evil, yielding again a contradiction.
\end{proof}
\begin{lem}
Let $e$ and $e'$ be inequivalent root edges in the support of a
rigid measure $m\in\mathcal{M}_{r}$, and let $f$ be an edge which
is a descendant of both e and $e'$. Consider a descendance path $A_{0}A_{1}\cdots A_{k}$
from $e$ to $f$, and a descendance path $B_{0}B_{1}\cdots B_{\ell}$
from $e'$ to $f$. Then $A_{k-1}=B_{\ell-1}$ and $A_{k}=B_{\ell}$.
\end{lem}
\begin{proof}
Assume to the contrary that $A_{k}=B_{\ell-1}$. The edge $f$ is
not equivalent to either $e$ or $e'$. Therefore there exist indices
$p,q$ such that $A_{p}A_{p+1}\not\to A_{p-1}A_{p}$ and $B_{q}B_{q+1}\not\to B_{q-1}B_{q}$.
It follows that \[
A_{p}A_{p+1}\cdots A_{k-1}B_{\ell-1}B_{\ell-2}\cdots B_{q}B_{q-1}\cdots B_{\ell-2}B_{\ell-1}A_{k-1}\cdots A_{p+1}A_{p}\]
is an evil loop, contradicting rigidity.
\end{proof}
These lemmas show that all the non-root edges in the support of a
rigid measure $m$ can be given an orientation. More precisely, given
a relation $e\Rightarrow f$ with $e$ a root edge, choose a descencence
path $A_{0}A_{1}\cdots A_{k}$ from $e$ to $f$, and assign $f$
the orientation $A_{k-1}A_{k}$. This will be called the \emph{orientation
of} $f$ \emph{away from the root edges.} Any common edge of two skeletons
in the support of $m$ can be oriented away from the root edges; indeed,
such an edge is not a root edge. In the proofs of the next two results,
we will be concerned with the descendants of a fixed root edge $e$,
and it will be convenient to orient the other root edges equivalent
to $e$ away from $e$. The edge $e$ can be oriented either way,
as needed.

\begin{lem}
\label{lem:two-edges}Fix a root edge in the support of a rigid measure
$m$, and suppose that two descendants $f=CX$ and $g=DX$ have orientations
pointing toward $X$. Then the turns $CXD$ and $DXC$ are not evil.
In particular, the angle between $f$ and $g$ is $60^{\circ}$, and
at least one of the edges $C'X,D'X$ opposite $f$ and $g$ has $m$
measure equal to zero.
\end{lem}
\begin{proof}
Assume to the contrary that either $CXD$ or $DXC$ is an evil turn.
The assumed orientations imply that $f\not\to g$ and $g\not\to f$.
Since one of the edges incident to $X$ must have measure zero, it
follows $f$ and $g$ are not collinear. Moreover, the edges $f'=XC'$
and $g'=XD'$ opposite to $f$ and $g,$ respectively, must be in
the support of $m$; in the contrary case we would have $f\to g$
or $g\to f$ if the angle between $f$ and $g$ is $120^{\circ}$,
or the turn $CXD$ would not be evil if the angle is $60^{\circ}$.
Let $A_{0}A_{1}\cdots A_{k}$ and $B_{0}B_{1}\cdots B_{\ell}$ be
descendance paths from $e$ to $f$ and $g$, respectively. By assumption,
we have $A_{k-1}=C$, $B_{\ell-1}=D$, and $A_{k}=B_{\ell}=X$. If
$A_{0}\ne B_{0}$, then $A_{0}=B_{1}$, $A_{1}=B_{0}$, and clearly
\[
A_{0}A_{1}\cdots A_{k}B_{\ell-1}B_{\ell-2}\cdots B_{1}\]
or its reverse is an evil loop, contradicting the rigidity of $m$.
Thus we must have $A_{0}=B_{0}$. Let $p$ be the largest integer
such that $A_{j}=B_{j}$ for $j\le p$. If $p<\min\{ k,\ell\}$, the
loop\[
A_{p}A_{p+1}\cdots A_{k}B_{\ell-1}B_{\ell-2}\cdots B_{p}\]
or its reverse is evil. Indeed, since $A_{p-1}A_{p}\to A_{p}A_{p+1}$
and $A_{p-1}A_{p}\to B_{p}B_{p+1}$, the turns $A_{p+1}B_{p}B_{p+1}$
and $B_{p+1}B_{p}A_{p+1}$ are both evil. We conclude that $p=\min\{ k,\ell\}$.
If $p=k$, it follows that the $B_{k-1}B_{k}\cdots B_{\ell}$ is a
descendance path from $f$ to $g$. Since $f\not\to g$, we must have
$B_{k+1}=C'$, and then the loop\[
B_{k}B_{k+1}\cdots B_{\ell}\]
or its reverse is obviously evil, leading to a contradiction. The
case $p=\ell$ similarly leads to a contradiction.
\end{proof}
\begin{thm}
Let $m\in\mathcal{M}_{r}$ be a rigid measure, and $e$ a root edge
in the support of $m$. Then the collection of all descendants of
$e$ is a skeleton.
\end{thm}
\begin{proof}
Since $m$ can be written as a sum of extremal measures, there exists
an extremal measure $m'\le m$ such that $m'(e)\ne0$. Since $f\to_{m}g$
implies that $f\to_{m'}g$, the support of $m'$ is a skeleton containing
all the descendants of $e$. Therefore it will suffice to show that
the descendants of $e$ form the support of some measure in $\mathcal{M}_{r}$.
We set $\mu(e)=1$, $\mu(f)=0$ if $f$ is not a descendant of $e$,
and for each descendant $f\ne e$ of $e$ we define $\mu(f)$ to be
the number of descendance paths from $e$ to $f$. Clearly, no edge
occurs twice in such a path; such an occurence would imply the existence
of an evil loop. Thus the number $\mu(f)$ is finite. To conclude
the proof, it suffices to show that $\mu\in\mathcal{M}_{r}$. The
support of $\mu$ is contained in the support of $m$. Therefore all
the branch points are in $\triangle_{r}$, and \[
\mu(A_{j}X{}_{j+1})=\mu(B_{j}Y{}_{j+1})=\mu(C_{j}Z{}_{j+1})=0\]
for all $j$. It remains to verify the balance conditions. Consider
a lattice point $X$ in $\triangle_{r}$. If no descendant of $e$
is incident to $X$, the six edges meeting at $X$ have $\mu$ measure
zero, and the balance condition is trivial. Otherwise, the number
of descendants of $e$ incident to $X$ can be $1,2,3,4$ or $5$;
the value $6$ is excluded by the rigidity of $m$. We first exclude
the case where this number is $1$. Assume indeed that there is only
one descendant incident to $X$, and let $A_{0}A_{1}\cdots A_{k}$
be a descendance path from $e$ with $A_{k}=X$. Since $A_{k-1}A_{k}$
has no descendants of the form $XY$, it follows that the turn $A_{k-1}A_{k}A_{k-1}$
is evil. On the other hand, since rigidity of $m$ insures that one
of the edges around $X$ has measure zero, the edge $A_{k-1}A_{k}$
is a strict decendant of some other edge $XZ$. In particular, $A_{k-1}A_{k}$
is not a root edge, and therefore it is not equivalent to $e$. It
follows that, for some $p$, we do not have $A_{p}A_{p+1}\to A_{p-1}A_{p}$,
and this implies that the turn $A_{p+1}A_{p}A_{p+1}$ is evil. Thus
\[
A_{p}A_{p+1}\cdots A_{k-1}A_{k}A_{k-1}\cdots A_{p+1}A_{p}\]
 is an evil loop, contrary to the rigidity of $m$.

Consider now the case when there are exactly two descendants of $e$
incident to $X$, call them $f$ and $g$. They cannot both point
away from $X$ since this would require the existence of a third descendant
pointing toward $X$. They cannot both point toward $X$. Indeed,
if this were the case, Lemma \ref{lem:two-edges} insures that one
of the edges opposite $f$ and $g$ has measure zero, and therefore
$f$ or $g$ has another descendant pointing away from $X$. Thus
we can assume that $f$ points toward $X$, and $g$ away from $X$,
in which case we have $f\to g$. Then $f$ and $g$ must be collinear,
and every descendance path for $g$ passes through $f$. Thus $\mu(f)=\mu(g)$,
which is the required balance condition.

Assume next that there are exactly three descendants incident to $X$.
Two of them must be noncollinear and of the form $WX\to XY$, and
the third descendant must be $WX\to XZ$, with the three edges forming
$120^{\circ}$ angles. Every descendance path for either $XY$ or
$XZ$ passes through $WX$, showing that $\mu(WX)=\mu(XY)=\mu(XZ)$,
and therefore satisfying the balance requirement at $X$.

Now, consider the case of exactly four descendants incident to $X$.
If these four edges form two collinear pairs, then two of them must
point towards $X$, and they will form an evil turn, contrary to Lemma
\ref{lem:two-edges}. Therefore we can find among the four descendants
two noncollinear edges $WX\to XY$, in which case we also have $WX\to XZ$
with these three edges forming $120^{\circ}$ angles. The fourth descendant
is not collinear with $WX$, so it makes a $60^{\circ}$ angle with
$WX$. If it points away from $X$, it must be a descendant of the
only incoming edge $WX$, and this is not possible. Therefore this
fourth edge must be $VX$ with $VX\to XY$ or $VX\to XZ$. Assume
$VX\to XY$ for definiteness. In this case, all descendance paths
for $XZ$ pass through $WX$, so that $\mu(XZ)=\mu(WX)$. On the other
hand, descendance paths for $XY$ pass either through $WX$ or through
$VX$, showing that $\mu(XY)=\mu(WX)+\mu(VX)$. The balance requirement
is again verified.

Finally, if $5$ descendants of $e$ are incident to $X$, then the
sixth edge must have mass equal to zero, and it is impossible to orient
the five edges so that every pair of incoming edges form a $60^{\circ}$
angle, and every outgoing edge is a descendant of an incoming edge.
Thus, this situation does not occur.
\end{proof}
The preceding proof shows that a rigid skeleton does not cross itself
transversely. In other words, a rigid skeleton does not contain four
edges meeting at the same point, such that they form two collinear
pairs of edges. The following figure shows the possible ways (up to
rotation) that the edges of a skeleton can meet, along with the possible
orientations. In each case, a (or the) dotted edge must have density
equal to zero.

\begin{center}

\includegraphics{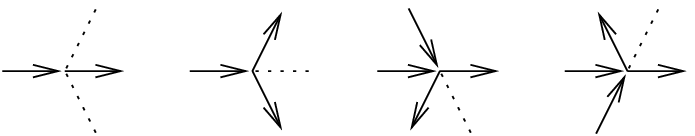}

\end{center}

The measure constructed in the above argument is the only measure
supported by the descendants of $e$ such that $\mu(e)=1$. We will
denote this measure $\mu_{e}$.

\begin{cor}
Let $m\in\mathcal{M}_{r}$ be a rigid measure, and let $e_{1},e_{2},\dots,e_{k}$
be a maximal collection of inequivalent root edges. Then we have\[
m=\sum_{j=1}^{k}m(e_{j})\mu_{e_{j}}.\]

\end{cor}
\begin{proof}
Let us say that an edge $f$ in the support of $m$ has height $\ge p\ge2$
if there exists a descendance path $A_{0}A_{1}\cdots A_{p}$ from
some root edge $e$ to $f$. If $f$ has height $\ge p$, we say that
$f$ has height equal to $p$ if it does not have height $\ge p+1$.
The requirement that $m\in\mathcal{M}_{r}$ shows that the measure
of any edge can be calculated in terms of the measures of edges of
smaller height, as can be seen from (\ref{eq:descendant-of-two}).
Therefore, $m$ is entirely determined by the values $m(e_{j})$,
$j=1,2,\dots k.$ On the other hand, the measure $m'=\sum_{j=1}^{k}m(e_{j})\mu_{e_{j}}$
has support contained in the support of $m$. Since $m'(e_{j})=m(e_{j}),$
we conclude that $m=m'$.
\end{proof}
We mention one more useful property of rigid skeletons.

\begin{lem}
\label{lemma:evil-paths}Let $e$ and $f$ be two edges in a rigid
skeleton. There exists a path $C_{0}C_{1}\cdots C_{p}$ in this skeleton
such that $C_{0}C_{1}=e$, $C_{p-1}C_{p}=f$ , and all the turns $C_{j-1}C_{j}C_{j+1}$
and $C_{j+1}C_{j}C_{j-1}$ are evil.
\end{lem}
\begin{proof}
The result is obvious if $e$ is a root edge. If it is not, choose
descendance paths $A_{0}A_{1}\cdots A_{k}$ from a root edge to $e$,
and $B_{0}B_{1}\cdots B_{\ell}$ from the same root edge to $f$.
If $A_{0}=B_{1}$, the path\[
A_{k}A_{k-1}\cdots A_{1}B_{1}B_{2}\cdots B_{\ell}\]
satisfies the requirements. If $A_{0}=B_{0}$ one chooses instead
the path\[
A_{k}A_{k-1}\cdots A_{r+1}B_{r}B_{r-1}\cdots B_{\ell},\]
where $r$ is the first integer such that $A_{r+1}\ne B_{r+1}$. If
no such integer exists, one of the paths is contained in the pther.
For instance, $k<\ell$ and $A_{j}=B_{j}$ for $j\le k$. In this
case the desired path is $B_{k-1}B_{k}\cdots B_{\ell}$.
\end{proof}
The path provided by this lemma is not generally a descendance path.

We conclude this section by introducing an order relation on the set
of skeletons contained in the support of a rigid measure $m$. Given
two rigid skeletons $S_{1}$ and $S_{2}$, we will write $S_{1}\prec_{0}S_{2}$
if $S_{1}$ has collinear edges $AX,XB$ and $S_{2}$ has collinear
edges $CX,XD$ such that $XA$ is $60^{\circ}$ clockwise from $XC$.
The following figure shows the four possible configurations of $S_{1}$
and $S_{2}$ around the point $X$, up to rotation, and assuming that
the two skeletons are contained in the support of a rigid measure.
The edges in $S_{1}\setminus S_{2}$ are dashed, the edges in $S_{2}\setminus S_{1}$
are solid without arrows, and the common edges are oriented away from
the root edges.

\begin{center}

\includegraphics[scale=0.8]{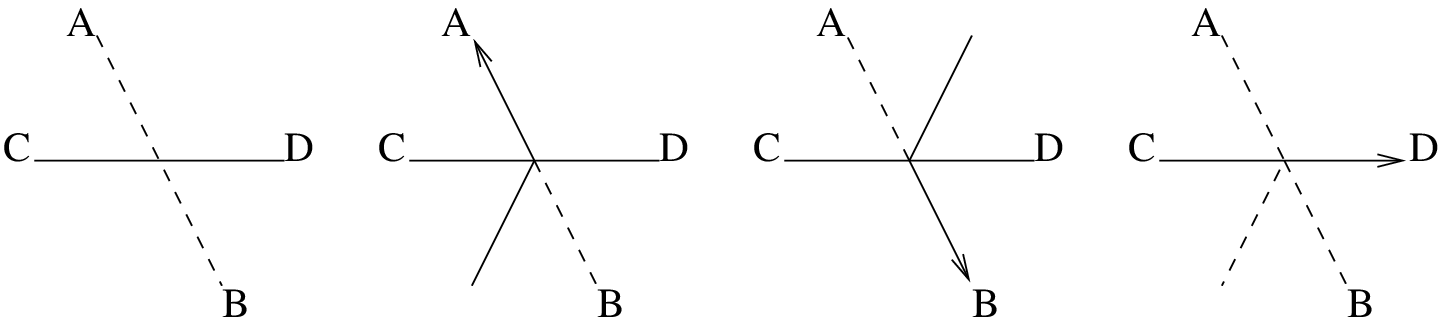}

\end{center}The four turns $AXC,AXD,BXC,$ and $BXD$ are evil. 

Note that the point $X$ could be on the boundary of $\triangle_{r}$,
but not one of the three corner vertices. It is possible that $S_{1}\prec_{0}S_{2}$
and $S_{2}\prec_{0}S_{1}$, as illustrated in the picture below (with
$S_{1}$ in dashed lines).\begin{center} \includegraphics[scale=0.7]{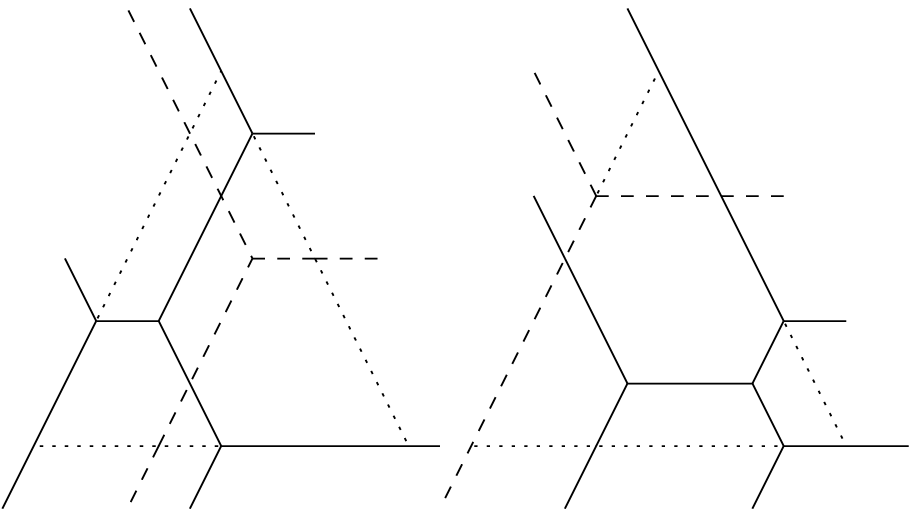}

\end{center} We will show that this does not occur when the skeletons are associated
with a fixed rigid measure.

\begin{thm}
Fix a rigid measure $m\in\mathcal{M}_{r}$ and an integer $n\ge1$.
There do not exist skeletons $S_{1},S_{2},\dots,S_{n}$ contained
in the support of $m$ such that\[
S_{1}\prec_{0}S_{2}\prec_{0}\cdots\prec_{0}S_{n}\prec_{0}S_{1}.\]

\end{thm}
\begin{proof}
Assume to the contrary that such skeletons exist. Choose for each
$j$ collinear edges $A_{j}X_{j},X_{j}B_{j}$ in the support of $S_{j}$
and collinear edges $C_{j}X_{j},X_{j}D_{j}$ in the support of $S_{j+1}$
(with $S_{n+1}=S_{1}$) such that $XA_{j}$ is $60^{\circ}$ clockwise
from $XC_{j}$. The rigidity of $m$ implies that one of the edges
$XY$ has measure zero. Therefore we must have either $X_{j}A_{j}\to X_{j}B_{j}$
or $X_{j}B_{j}\to X_{j}A_{j}$. Label these two edges $f_{j}$ and
$f_{j}'$ so that $f_{j}\to f'_{j}$, and note that $f_{j}$ is not
the descendant of any edge $XY$, except possibly $f'_{j}$. Analogously,
denote the edges $X_{j-1}C_{j-1}$ and $X_{j-1}D_{j-1}$ by $e_{j}$
and $e'_{j}$ so that $e_{j}'\to e{}_{j}$. Since both $e_{j}$ and
$f_{j}$ are contained in $S_{j}$, Lemma \ref{lemma:evil-paths}
provides a path with evil turns joining $e_{j}$ and $f_{j}$. A moment's
thought shows that this path either begins at $X_{j-1}$, or it begins
with one of $C_{j-1}X_{j-1}D_{j-1},D_{j-1}X_{j-1}C_{j-1}$. If the
second alternative holds, remove the first edge from the path. Performing
the analogous operation at the other endpoint, we obtain a path $\gamma_{j}$
with only evil turns which starts at $X_{j-1}$ ends at $X_{j}$ (with
$X_{j-1}=X_{n}$ if $j=0$), its first edge is one of $X_{j-1}C_{j-1},X_{j-1}D_{j-1}$,
and its last edge is one of $A_{j}X_{j},B_{j}X_{j}$. As noted above,
the turn formed by the last edge of $\gamma_{j}$ and the first edge
of $\gamma_{j+1}$ is evil. We conclude that the loop $\gamma_{1}\gamma_{2}\cdots\gamma_{n}$
is evil, contradicting rigidity.
\end{proof}
The preceding result shows that there is a well-defined order relation
on the set of skeletons in the support of a rigid measure $m$ defined
as follows: $S\prec S'$ if there exist skeletons $S_{1},S_{2},\dots,S_{k}$
such that \[
S=S_{1}\prec_{0}S_{2}\prec_{0}\cdots\prec_{0}S_{k}=S'.\]
The following figure shows the support of a rigid measure $m\in\mathcal{M}_{6}$.
The elements of a maximal collection of mutually inequivalent root
edges have been indicated with dots.

\begin{center}

\includegraphics{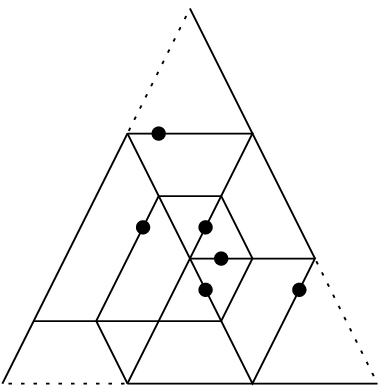}

\end{center}

For this measure, there is a smallest skeleton (relative to $\prec$)
pictured below. The reader can easily draw all the other skeletons
and determine the order relation.\begin{center}

\includegraphics{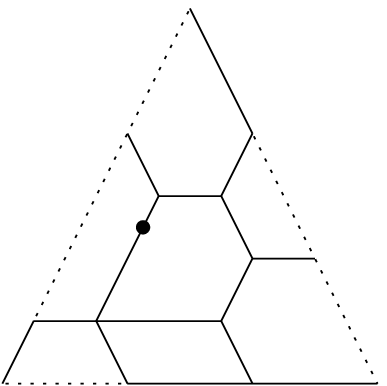}\end{center}

\section{Horn Inequalities and Clockwise Overlays}

The results of \cite{KT,KTW} show that the triples $(\alpha,\beta,\gamma)=\partial m$
with $m\in\mathcal{M}_{r}$ are precisely those triples of increasing
nonegative vectors in $\mathbb{R}^{r}$ with the property that there
exist selfadjoint matrices $A,B,C\in M_{r}(\mathbb{C})$ such that
\[
\lambda_{A}(\ell)=\alpha_{r+1-\ell},\lambda_{B}(\ell)=\beta_{r+1-\ell},\lambda_{C}(\ell)=\gamma_{r+1-\ell},\quad\ell=1,2,\dots,r,\]
and $A+B+C$ is a multiple of the identity, namely, $2\omega(m)1_{\mathbb{C}^{r}}$.
The Horn inequalities for these matrices are \[
\sum_{i\in I}\lambda_{A}(i)+\sum_{j\in J}\lambda_{B}(j)+\sum_{k\in K}\lambda_{C}(k)\le2s\omega(m)\]
when $I,J,K\subset\{1,2,\dots,r\}$ have cardinality $s$ and $c_{IJK}>0$.
Applying this inequality to the matrices $-A,-B,-C$ instead and switching
signs, we obtain\[
\sum_{i\in I}\lambda_{A}(r+1-i)+\sum_{j\in J}\lambda_{B}(r+1-j)+\sum_{k\in K}\lambda_{C}(r+1-k)\ge2s\omega(m)\]
Equivalently,\[
\sum_{i\in I}\alpha_{i}+\sum_{j\in J}\beta_{j}+\sum_{k\in K}\gamma_{k}\ge2s\omega(m).\]
Now, the sets $I,J,K$ are obtained from some measure $\nu\in\mathcal{M}_{s}$
with weight $r-s$, and we will see how this inequality follows from
the superposition of the support of $m$ and the puzzle associated
with $\nu$. Let $h$ be the honeycomb provided by Lemma \ref{lem:Honeycomb-from-m}.
Let $D\subset\triangle_{r}$ be a region bounded by small edges, and
let $X_{j}Y_{j},j=1,2,\dots,p$ be an enumeration of the edges of
$\partial D$, oriented so that $D$ lies on the left of $X_{j}Y_{j}$.
For each $j$, there is $\varepsilon_{j}=\varepsilon_{X_{j}Y_{j}}=\pm1$
such that $Y_{j}-X_{j}\in\{\varepsilon_{j}u,\varepsilon_{j}v,\varepsilon_{j}w\}$.
The definition of honeycombs implies then the identity\[
\sum_{XY\subset\partial D}\varepsilon_{XY}h(XY)=\sum_{j=1}^{p}\varepsilon_{j}h(X_{j}Y_{j})=0,\]
which is easily deduced by induction on the size of $D$. We would
like to verify that the sum\[
S=\sum_{i\in I}h(A_{i-1}A_{i})+\sum_{j\in J}h(B_{j-1}B_{j})+\sum_{k\in K}h(C_{k-1}C_{k})\]
is nonnegative, where $I,J,K$ are given by a measure $\nu\in\mathcal{M}_{s}$
with $\omega(\nu)=r-s$. The inflation of the measure $\nu$ yields
a partition of $\triangle_{r}$ into white pieces (the translated
parts of $\triangle_{s}$), gray parallelograms, and light gray pieces.
Denote by $D$ the union of the gray parallelograms $P_{1},P_{2},\dots,P_{\sigma}$
and white pieces $W_{1},W_{2},\dots,W_{\tau}$. The edges $A_{i-1}A_{i},i\in I$,
$B_{j-1}B_{j},j\in J$, and $C_{k-1}C_{k},k\in K$ are contained in
the intersection $\partial D\cap\partial\triangle$. Since\[
\sum_{XY\subset\partial D}\varepsilon_{XY}h(XY)=\sum_{XY\subset\partial W_{\ell}}\varepsilon_{XY}h(XY)=0,\quad\ell=1,2,\dots,\tau,\]
and\[
S=\sum_{{\rm white}\, XY\subset\partial D\cap\partial\triangle_{r}}\varepsilon_{XY}h(XY),\]
we have\begin{eqnarray*}
S & = & S-\sum_{XY\subset\partial D}\varepsilon_{XY}h(XY)+\sum_{\ell=1}^{\tau}\sum_{XY\subset\partial W_{\ell}}\varepsilon_{XY}h(XY)\\
 & = & -\sum_{\ell=1}^{\sigma}\sum_{\text{gray }XY\subset\partial P_{\ell}}\varepsilon_{XY}h(XY),\end{eqnarray*}
where the last sums are taken over the light gray edges of $P_{\ell}$.
Condition (3) implies that\[
-\sum_{\text{gray }XY\subset\partial P_{\ell}}\varepsilon_{XY}h(XY)=\sum_{e}m(e),\]
 where the sum is taken over all small edges contained in $P_{\ell}$
which are not parallel to the sides of $P_{\ell}$. This immediately
implies the desired Horn inequality, and it also tells us when equality
is attained: this happens if and only if all the edges $e\subset P_{\ell}$
for which $m(e)>0$ are parallel to the edges of $P_{\ell}$, $\ell=1,2,\dots,s$.
In other words, the support of $m$ must cross each $P_{\ell}$ along
lines parallel to the edges of $P_{\ell}$. The following figure illustrates
the support of a measure $\nu$, the inflation of $\nu$, and the
support of a measure $m$ which satisfies the Horn equality associated
with $\nu$. In this example the support of the measure $m$ never
crosses the white puzzle pieces. \begin{center}

\includegraphics[scale=0.5]{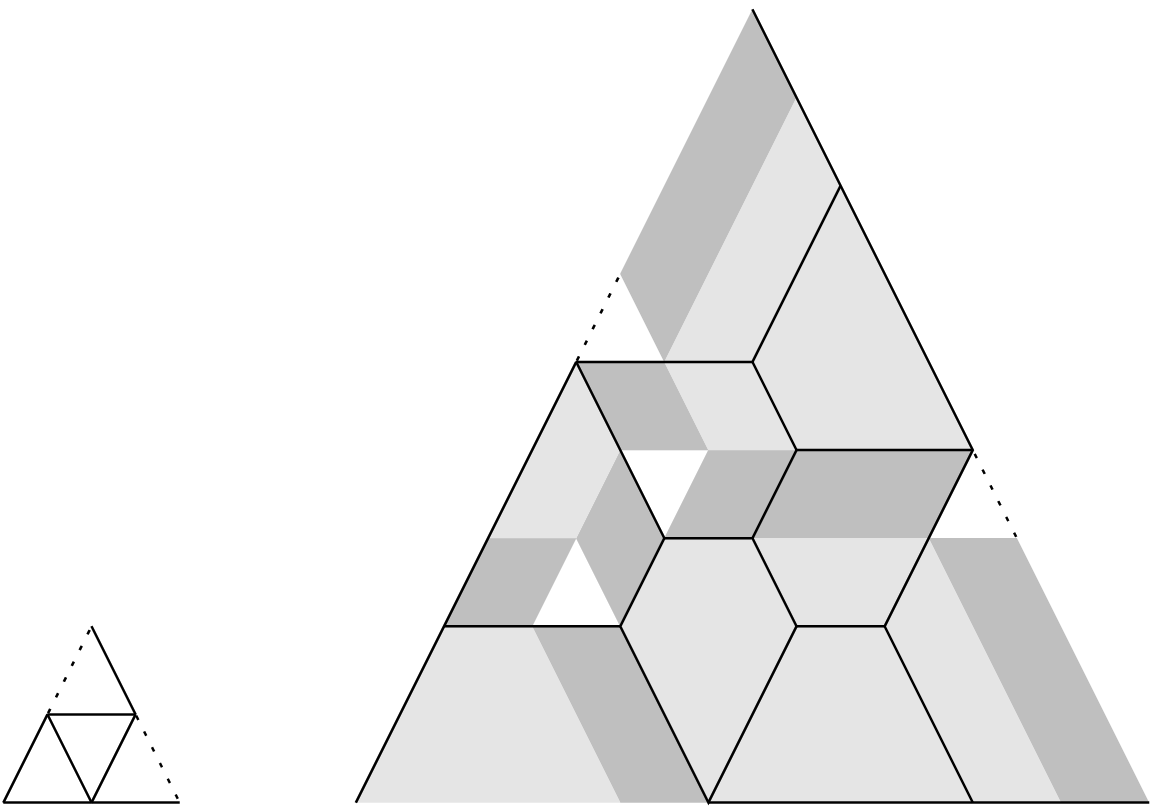}

\end{center}The next example involves basically the mirror image of the measure
$m$, and its support never crosses the light gray pieces. \begin{center}

\includegraphics[scale=0.5]{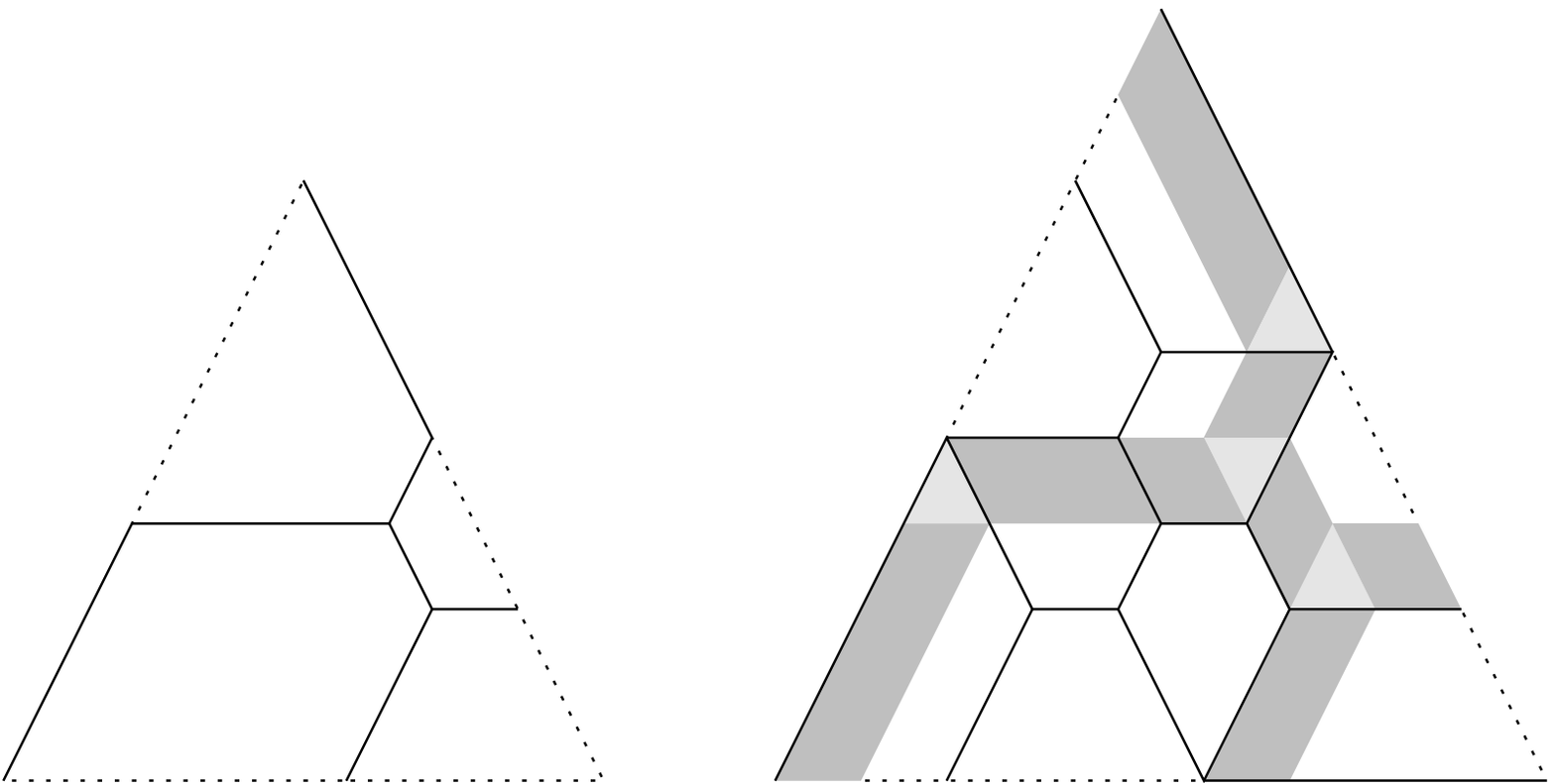}

\end{center}This phenomenon is related with the fact that the support of $m$
is actually a rigid skeleton in both cases. A rigid skeleton does
not cross itself transversely. Thus, in a case of equality, it cannot
have any branch points in the interior of a gray parallelogram. It
follows then that the skeleton always crosses these parallelograms
between white pieces or between gray pieces, but not both.

Assume that we are in a case of equality\[
\sum_{i\in I}h(A_{i-1}A_{i})+\sum_{j\in J}h(B_{j-1}B_{j})+\sum_{k\in K}h(C_{k-1}C_{k})=0.\]
 In this case, we can define a measure $\mu\in\mathcal{M}_{s}$ by
moving the support of $m$ to $\triangle_{s}$ in the following way:
those parts which are contained in white puzzle pieces are simply
translated back to $\triangle_{s}$ (along with their densities);
the segments in the support of $m$ which cross between white pieces
are deleted; the segments which cross between light gray pieces are
replaced by the coresponding parallel sides of white pieces, and the
density is preserved. It may be that several segments cross between
light gray pieces, in which case the density of the corresponding
side of a white piece is the sum of their densities. When the measure
$\mu$ can be obtained using this procedure, we will say that $\mu$
is obtained by \emph{contracting} $m$, and that $\mu$ is \emph{clockwise}
from $\nu$ (or that $(\mu,\nu)$ form a \emph{clockwise overlay})\emph{.}
This is easily seen to be an extension of the notion of clockwise
overlay introduced in \cite{KTW} (see also item (1), second case,
in the proof of Theorem \ref{thm:stretching}). Generally, a clockwise
overlay $(\mu,\nu)$ can be obtained by shrinking more than one measure
$m$. Indeed, the shrinking operation loses all the information about
the branch points of $m$ in the light gray puzzle pieces.

In the first case illustrated above, the support of the measure $\mu$
is actually contained in the support of $\nu$; this is what happens
when the support of $m$ does not cross the white pieces. In the second
case illustrated above we obtain the following figure for the supports
of $\nu$ and $\mu$. \begin{center}

\includegraphics[scale=0.5]{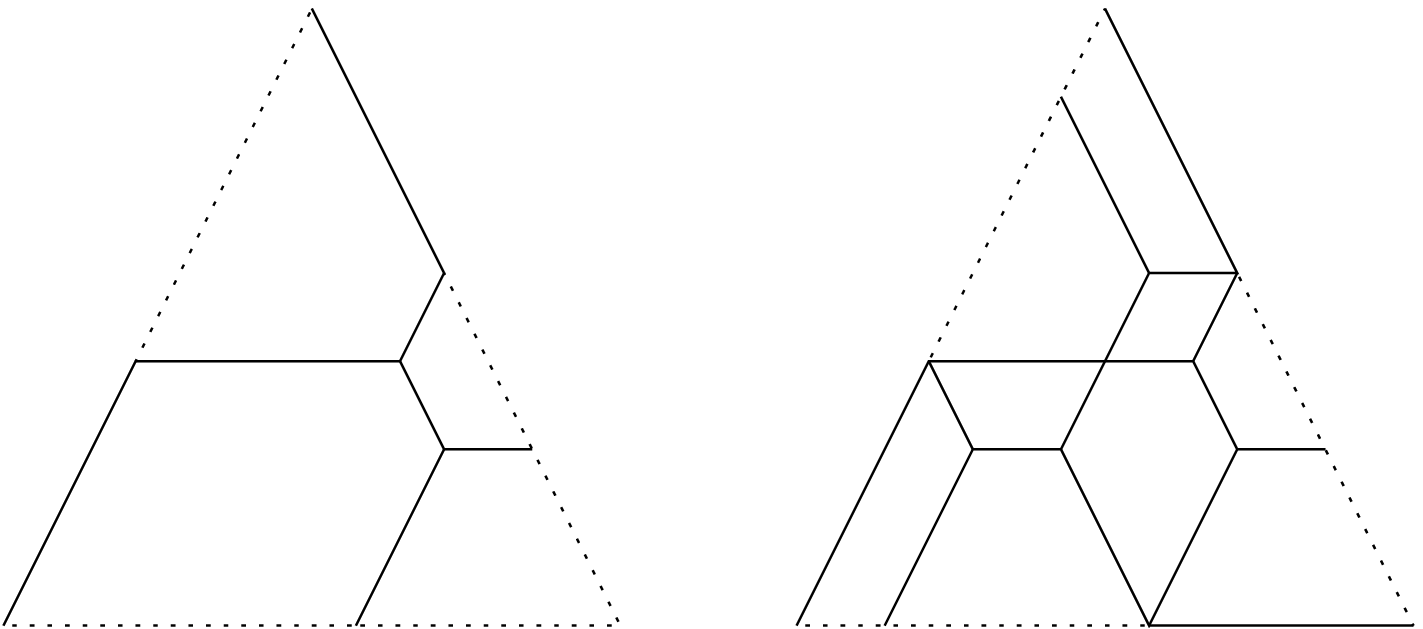}

\end{center} 

We will need one more important property of clockwise overlays.

\begin{prop}
\label{prop:support-of-clockwise-measure}Let $(\mu,\nu)$ be a clockwise
overlay obtained by contracting a measure $m$. Then $\omega(\mu)=\omega(m)$,
and $\mu^{*}\le m^{*}$.
\end{prop}
\begin{proof}
Consider the puzzle obtained by inflating the measure $\varepsilon\nu$
for $\varepsilon>0$. The white pieces of the puzzle are independent
of $\varepsilon$. Since the support of $m$ intersects any gray parallelogram
in the puzzle of $\nu$ only on segments parallel to the edges of
the parallelogram, it follows that there exists a measure $m_{\varepsilon}$
obtained by translating the support of $m$ in each white piece, and
applying appropriate translation and/or shrinking in the gray parallelograms
and light gray puzzle pieces. Clearly $m_{1}=m$, and all the measures
$m_{\varepsilon}$ are homologous to $m$; in fact, homologous sides
have equal densities, and therefore $\omega(m)=\omega(m_{\varepsilon})$
for all $\varepsilon>0$. (Here it may be useful to recall that $\omega(m)$
is defined in terms of its densities outside $\triangle_{r}$, and
the ouside edges are not generally present in our drawings.) Moreover,
all the measures $m_{\varepsilon}^{*}$ have the same support, except
that some of the densities are decreased for $\varepsilon<1$. The
measure $\mu$ is simply the limit of $m_{\varepsilon}$ as $\varepsilon\to0$,
and the statement follows immediately from this observation.

The following pictures illustrates the process as applied to the above
examples for $\varepsilon=2/3$ and $\varepsilon=1/3$.\begin{center}

\includegraphics[scale=0.5]{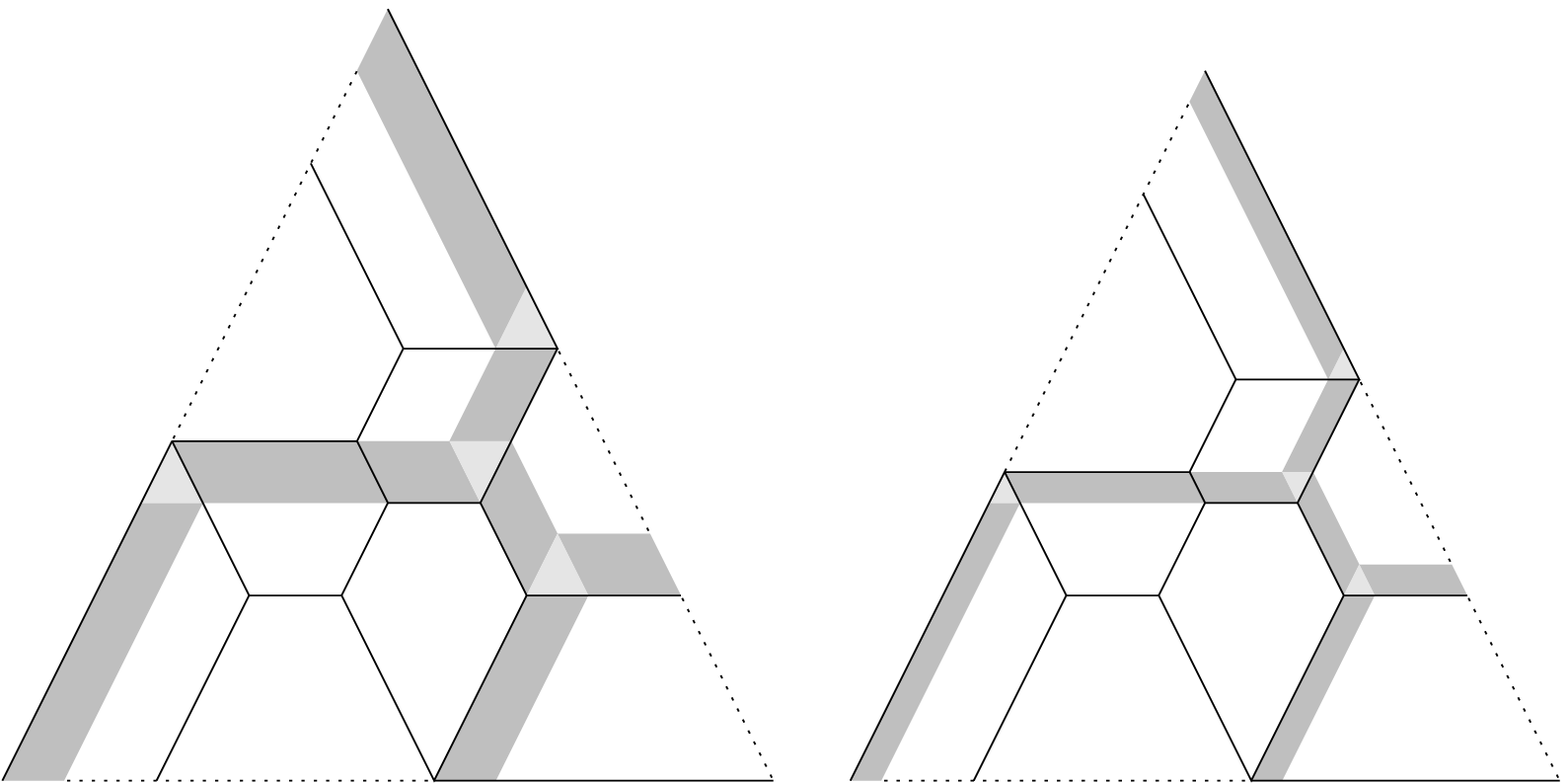}

\end{center}\begin{center}

\includegraphics[scale=0.5]{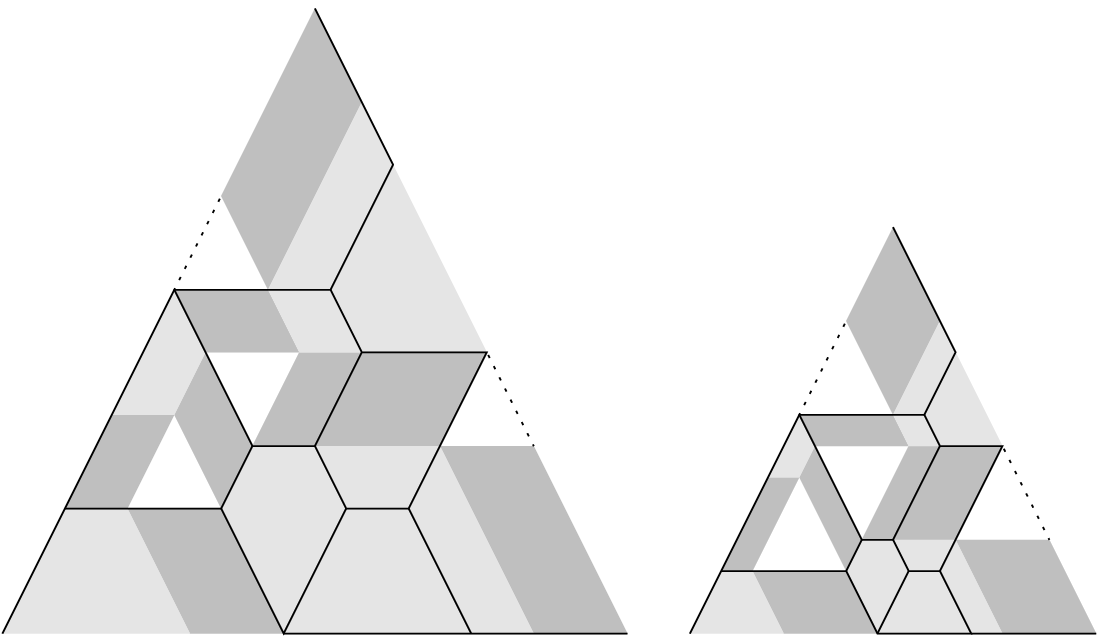}

\end{center} 
\end{proof}
Unfortunately, the definition of clockwise overlays is not quite explicit
since they are seen as the result of a process --- something akin
to defining a car as the end product of car manufacture. We can however
use the relation $\prec_{0}$ between skeletons to produce an important
class of clockwise overlays.

\begin{thm}
\label{thm:stretching}Let $\mu_{1},\mu_{2}\in\mathcal{M}_{r}$ be
such that $\mu_{1}+\mu_{2}$ is rigid, the support $S_{j}$ of $\mu_{j}$
is a skeleton, and $S_{2}\not\prec_{0}S_{1}$. Then $(\mu_{1},\mu_{2})$
is a clockwise overlay.
\end{thm}
\begin{proof}
We need to inflate $\mu_{2}$, and construct a measure $m_{1}$ such
that $\mu_{1}$ is obtained from $m_{1}$ by the shrinking process
described above. It is clear what the measure $m_{1}$ should be on
the interior of every white puzzle piece. The common edges of $S_{1}$
and $S_{2}$ cannot be root edges; orient them away from the root
edges, and attach them (along with their $\mu_{1}$ masses) to the
white puzzle piece on their right side. What remains to be proved
is that this partialy defined measure can be extended so as to satisfy
the balance condition at all points. For this purpose we only need
to analyze the situation at lattice points where $S_{1}$ and $S_{2}$
meet. For each such lattice point, there will be $2,3,$ or $4$ edges
of each skeleton meeting at that point, and this gives rise to many
possibilities. In order to reduce the number of cases we need to study,
observe that the inflation construction is invariant relative to rotations
of $60^{\circ}$, and therefore the position (but perhaps not the
orientation) of the edges in one of the skeletons can be fixed. In
the following figures, the arrows indicate the orientation on the
edges in $S_{1}\cap S_{2}$. The other edges of $S_{1}$ are dashed,
and the other edges of $S_{2}$ are solid without arrows. In each
case, the extension required after inflation is indicated by dashed
lines crossing (or on the boundary of) parallelogram pieces. In he
following enumeration, the label $(p,q)$ signifies that $S_{1}$
has $p$ and $S_{2}$ has $q$ edges meeting at one point.
\begin{enumerate}
\item $(2,2)$ The edges of the skeletons may overlap, and after a rotation
the orientation is as in the figure below.\begin{center}

\includegraphics[scale=0.5]{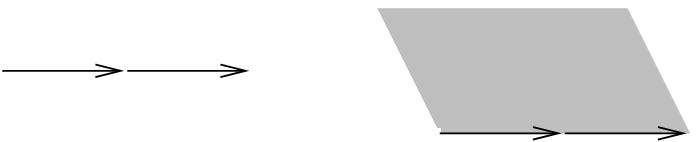}

\end{center}No extensions are required in this case. If the skeletons do not overlap,
we have two possibilities:\begin{center}

\includegraphics[scale=0.5]{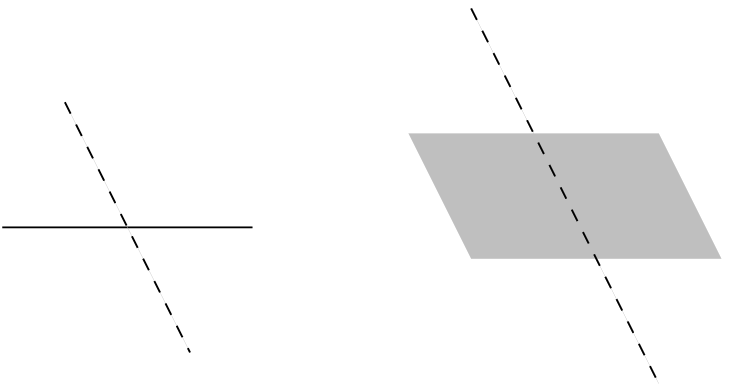}

\end{center}and finally\begin{center}

\includegraphics[scale=0.5]{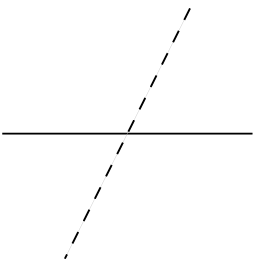}

\end{center}which would imply $S_{2}\prec_{0}S_{1}$, contrary to the hypothesis.
\item $(3,2)$ In this case there is (up to rotations) only the case illustrated
in the figure.\begin{center}

\includegraphics[scale=0.5]{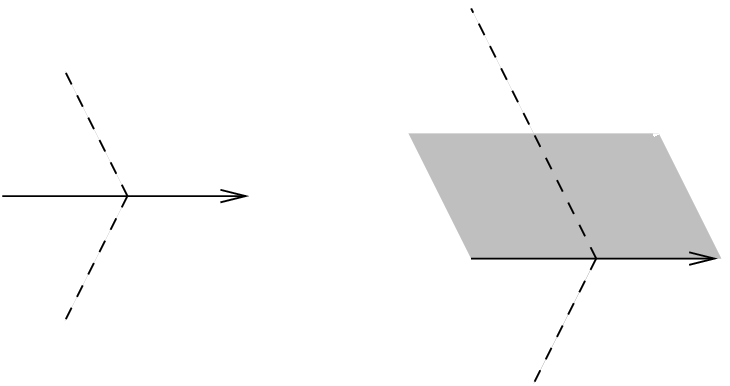}

\end{center}
\item $(4,2)$ Up to rotations, there are three possibilities. In the first
one we have an extension as shown.\begin{center}

\includegraphics[scale=0.5]{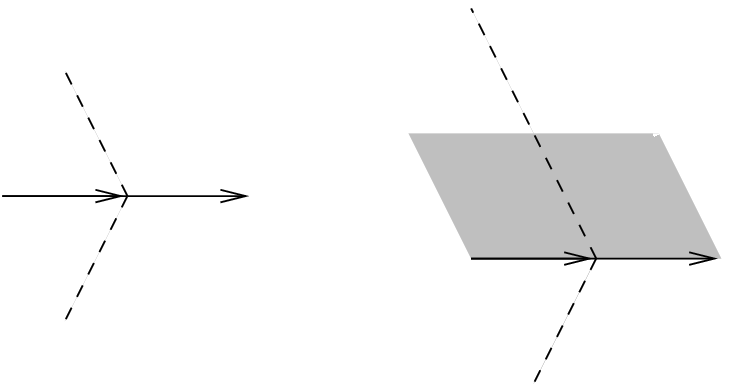}

\end{center}The orientation shown above is the only one which is compatible with
the rigidity of $m$. In the following figure, the orientation given
is also the only possible one.\begin{center}

\includegraphics[scale=0.5]{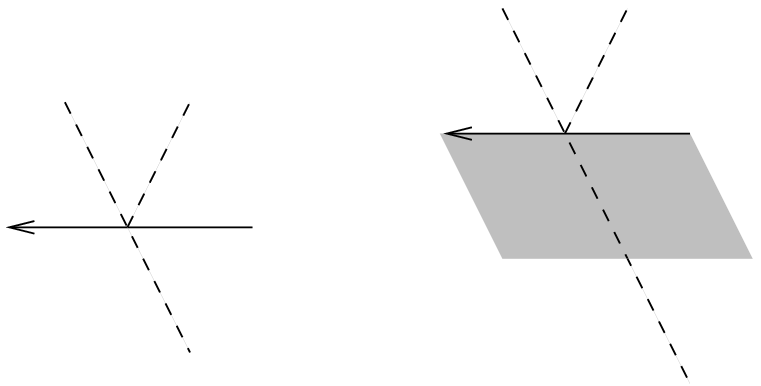}

\end{center}The third situation\begin{center}

\includegraphics[scale=0.5]{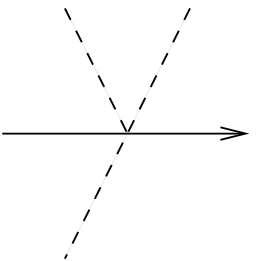}

\end{center}implies $S_{2}\prec S_{1}$.
\item (2,3) The case illustrated is the only one up to rotations.\begin{center}

\includegraphics[scale=0.5]{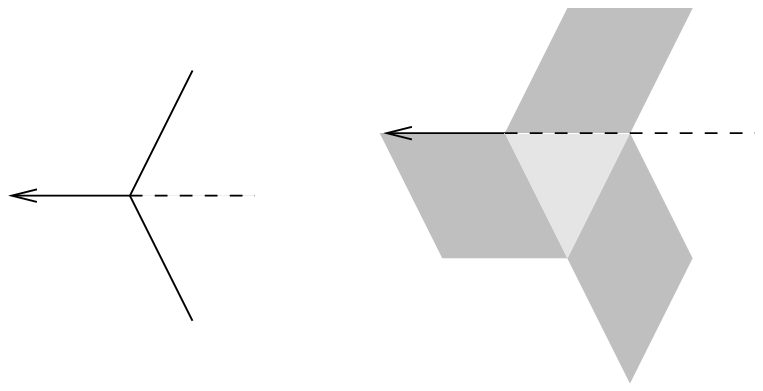}

\end{center}
\item $(3,3)$ There are two cases up to rotations.\begin{center}

\includegraphics[scale=0.5]{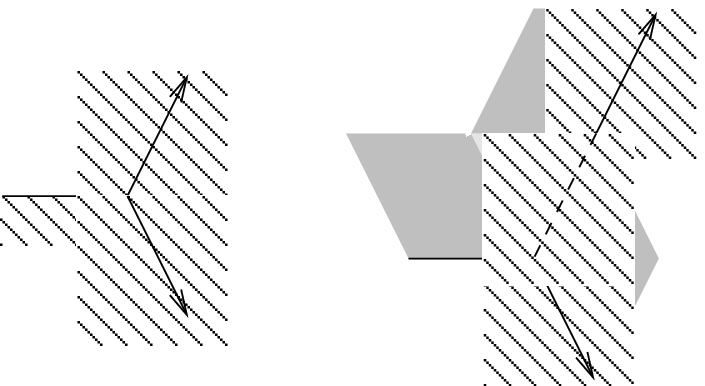}

\end{center}The second case is not compatible with rigidity.\begin{center}

\includegraphics[scale=0.5]{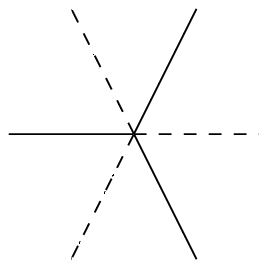}

\end{center}
\item $(4,3)$ There is only one position of $S_{1}$ compatible with rigidity,
but there are two possible orientations. \begin{center}

\includegraphics[scale=0.5]{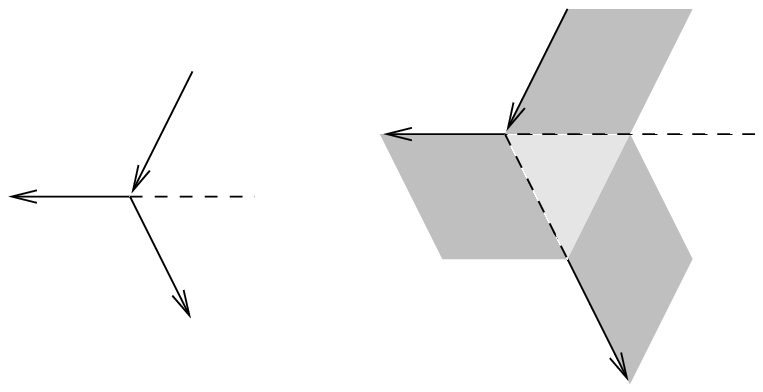}

\end{center}The second orientation requires a different extension.\begin{center}

\includegraphics[scale=0.5]{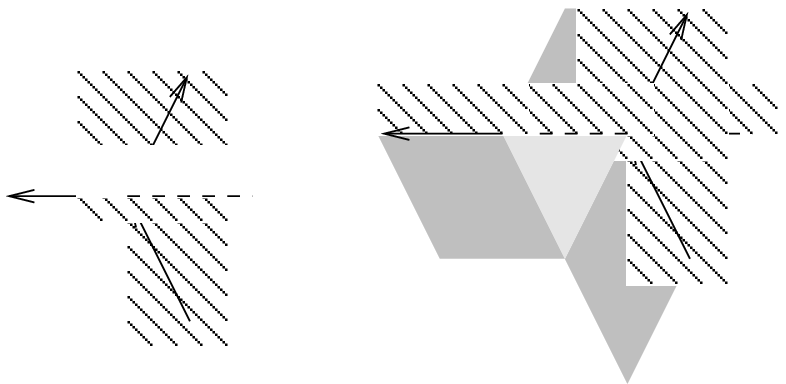}

\end{center}
\item $(2,4)$ There are three possibilities up to rotation.\begin{center}

\includegraphics[scale=0.5]{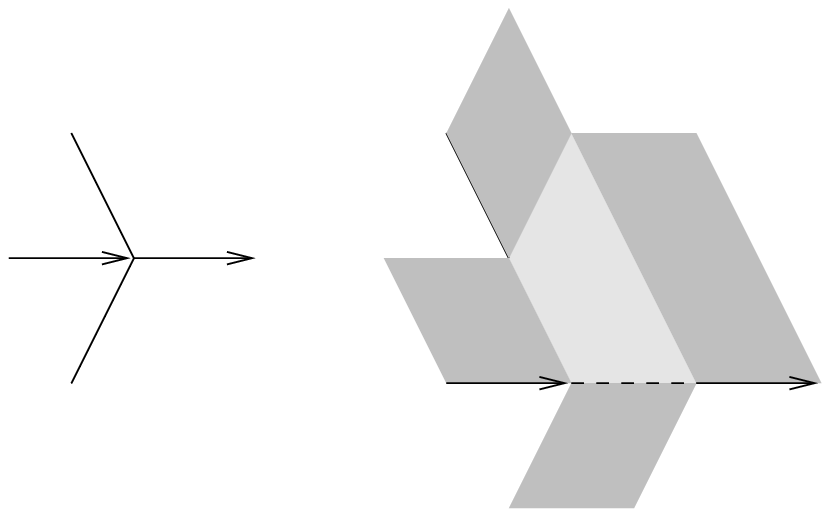}

\end{center}In the figure above, there is no ambiguity in the orientation. For
the illustration we assigned $\mu_{2}$ masses of $1$ and $2$ to
the edges.\begin{center}

\includegraphics[scale=0.5]{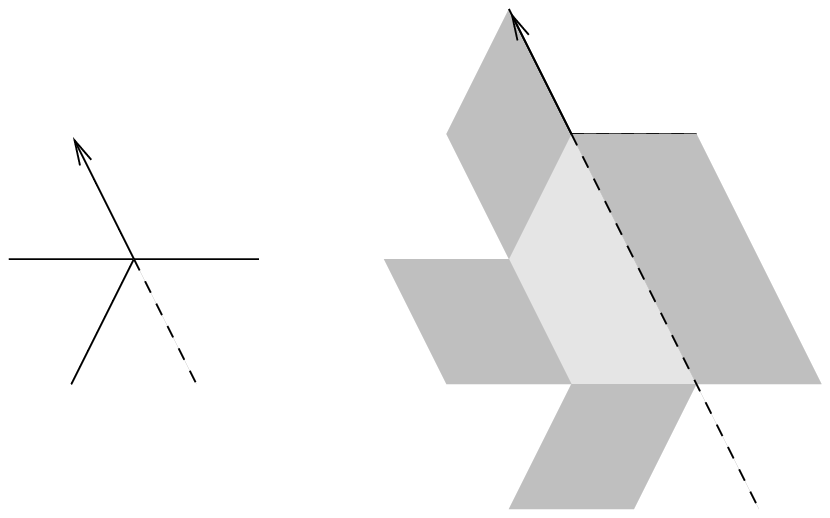}

\end{center}The orientation is also clear in this case. The third case implies
$S_{2}\prec S_{1}$.\begin{center}

\includegraphics[scale=0.5]{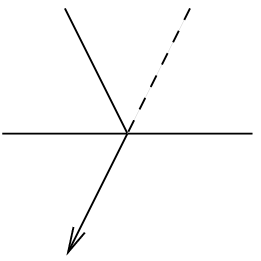}

\end{center}
\item $(3,4)$ There is only one position compatible with rigidity, and
there are only two possible orientations.\begin{center}

\includegraphics[scale=0.5]{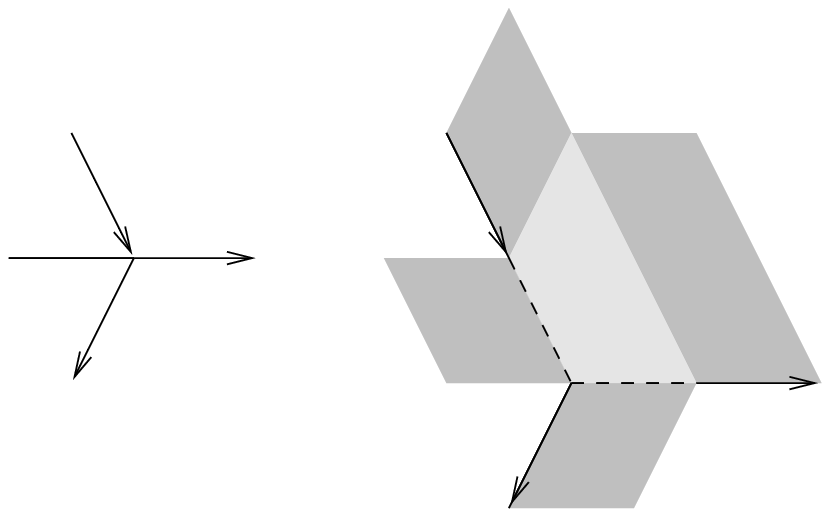}

\end{center}\begin{center}

\includegraphics[scale=0.5]{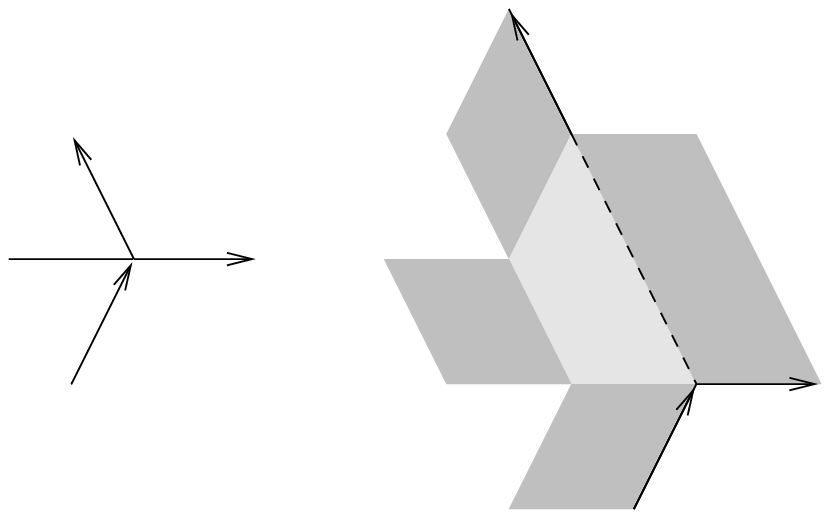}

\end{center}
\item $(4,4)$ When the two skeletons overlap completely, there are two
possible orientations.\begin{center}

\includegraphics[scale=0.5]{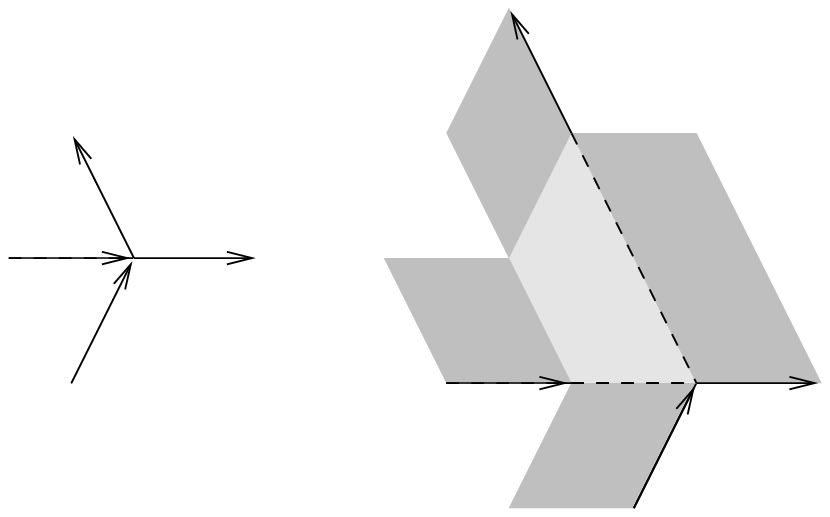}

\end{center}\begin{center}

\includegraphics[scale=0.5]{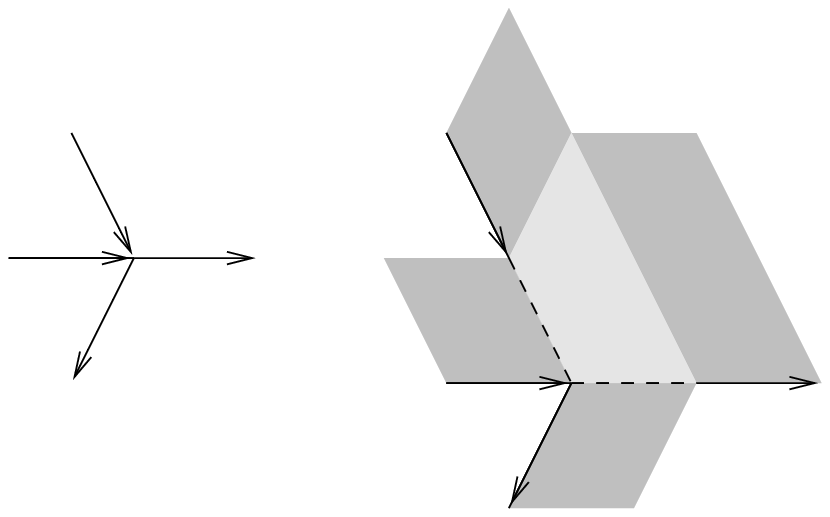}

\end{center}When the skeletons do not overlap completely, there is only one relative
position of the two skeletons which is compatible both with rigidity
and with $S_{2}\not\prec S_{1}$. There is only one possible orientation.\begin{center}

\includegraphics[scale=0.5]{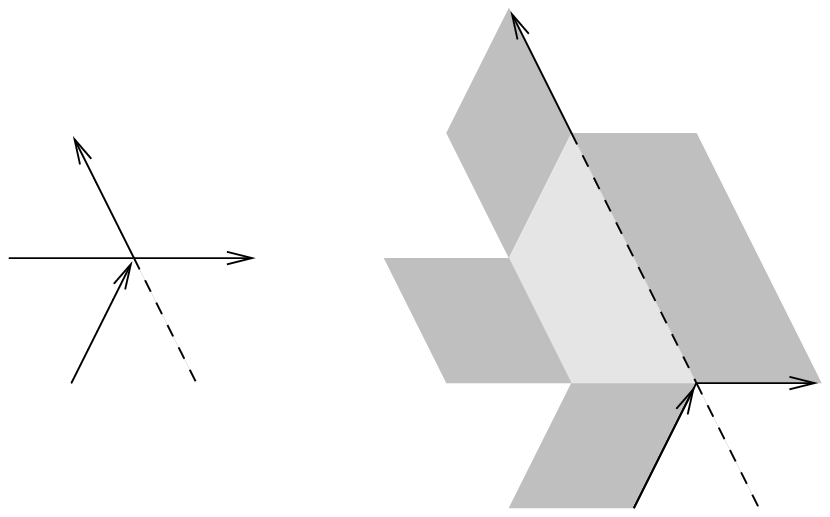}

\end{center}
\end{enumerate}
\end{proof}
The following result follows easily by induction, inflating successively
the measures $\mu_{p},\mu_{p-1},\dots,\mu_{t+1}$.

\begin{cor}
Let $m\in\mathcal{M}_{r}$ be a rigid measure, and write it as $m=\sum_{\ell=1}^{p}\mu_{\ell}$,
where $\mu_{\ell}$ is supported on the skeleton $S_{\ell}$. Assume
also that $S_{i}\prec S_{j}$ implies that $i\le j$. Then the pair
$\left(\sum_{\ell=1}^{t}\mu_{\ell},\sum_{\ell=t+1}^{p}\mu_{\ell}\right)$
is a clockwise overlay for $1\le t<p$.
\end{cor}
For the clockwise overlays $(\mu_{1},\mu_{2})$ considered in the
preceding two results there is a canonical construction for the measure
$m_{1}\in\mathcal{M}_{r+\omega(\mu_{2})}$. We will call this measure
$m_{1}$ the \emph{stretch of} $\mu_{1}$ \emph{to the puzzle} \emph{of}
$\mu_{2}$.

\section{Proof of the Main Results}

Fix a triple $(I,J,K)$ of subsets with cardinality $r$ of $\{0,1,\dots,n\}$
such that $c_{IJK}=1$, and let $m\in\mathcal{M}_{r}$ be the corresponding
measure. It will be convenient now to use the normalization $\tau(1)=n$
in a finite factor. This will not require renormalizations when passing
to a subfactor, and has the added benefit of working in finite dimensions
as well. In order to prove the intersection results in the introduction,
we will want to prove the following related properties:

\textbf{Property $A(I,J,K)$ or $A(m)$}. Given a II$_{1}$ factor
$\mathcal{A}$ with $\tau(1)=n$, and given flags $\mathcal{E},\mathcal{F},\mathcal{G}$
with $\tau(E_{\ell})=\tau(F_{\ell})=\tau(G_{\ell})=\ell$, $\ell=0,1,2,\dots,n$,
the intersection\[
S(\mathcal{E},I)\cap S(\mathcal{F},J)\cap S(\mathcal{G},K)\]
is not empty.

\textbf{Property $B(I,J,K)$ or $B(m).$} There exists a lattice polynomial
$p\in L(\{ e_{j},f_{j},g_{j}:1\le j\le n\})$ with the following property:
for any finite factor $\mathcal{A}$ with $\tau(1)=n$, and for generic
flags $\mathcal{E}=(E_{j})_{j=0}^{n}$, $\mathcal{F}=(F_{j})_{j=0}^{n},$
$\mathcal{G}=(G_{j})_{j=0}^{n}$ such that $\tau(E_{j})=\tau(F_{j})=\tau(G_{j})=j$,
the projection $P=p(\mathcal{E},\mathcal{F},\mathcal{G})$ has trace
$\tau(P)=r$ and, in addition\[
\tau(P\wedge E_{i})=\tau(P\wedge F_{j})=\tau(P\wedge G_{k})=\ell\]
when $i_{\ell}\le i<i_{\ell+1},j_{\ell}\le j<j_{\ell+1},k_{\ell}\le k<k_{\ell+1}$
and $\ell=0,1,\dots,r$, where $i_{0}=j_{0}=k_{0}=0$ and $i_{r+1}=j_{r+1}=k_{r+1}=n+1$.

We will prove these properties by reducing them to simpler measures
for which they are trivial. The basic reduction is from an arbitrary
measure to a skeleton.

\begin{prop}
\label{prop:reduction-of-AB(m)}Let $m\in\mathcal{M}_{r}$ a rigid
measure, and write $m=\sum_{\ell=1}^{p}\mu_{\ell}$, where $\mu_{\ell}$
is supported by a skeleton $S_{\ell}$ , and $S_{i}\prec S_{j}$ implies
$i\le j$. Let $\widetilde{\mu_{1}}\in\mathcal{M}_{\widetilde{r}}$,
$\widetilde{r}=\sum_{\ell=2}^{p}\omega(\mu_{\ell})$ be the stretch
of $\mu_{1}$ to the puzzle of $m'=\sum_{\ell=2}^{p}\mu_{\ell}$.
If $A(\widetilde{\mu_{1}})$ and $A(m')$ $($resp., $B(\widetilde{\mu_{1}})$
and $B(m'))$ are true, then $A(m)$ $($resp., $B(m))$ is true as
well.
\end{prop}
\begin{proof}
With the usual notation $A_{i}=iu$, $X{}_{i}=A_{i}+w$, the edges
$A_{i}X{}_{i}$ are oriented in the direction of $w$ (if they belong
to the support of $m$). Let us set $a_{i}=m(A_{i}X{}_{i})$, $a_{i}^{(1)}=\mu_{1}(A_{i}X{}_{i})$,
and $a'_{i}=m'(A_{i}X{}_{i})$, so that $a_{i}=a'_{i}+a_{i}^{(1)}$,
and $n=r+\sum_{i=0}^{r}a_{i}$. The measure $\widetilde{\mu_{1}}$
is associated with the triangle $\triangle_{r_{1}}$, where $r_{1}=r+\sum_{i=0}^{r}a'_{i}.$
Its support may intersect the left side of this triangle only at the
points $A_{\ell(i)}$, where $\ell(0)=0$, and $\ell(i)=i+\sum_{s=0}^{i-1}a'_{s}$
for $i>0$; this follows from the way the inflation of $m'$ is constructed,
and from the outward orientation of the segments $A_{i}X{}_{i}$.
Moreover, $\widetilde{\mu_{1}}(A_{\ell(i)}X{}_{\ell(i)})=a_{i}^{(1)}$
for $i=0,1,\dots,r$. Denote by $I^{(1)},J^{(1)},K^{(1)}\subset\{1,2,\dots,n\}$
the sets determined by the measure $\widetilde{\mu_{1}}$, and by
$I',J',K'\subset\{1,2,\dots,n_{1}\}$ those corresponding with $m'$,
where we set $n_{1}=n-\omega(m_{1})$. For instance, we have\[
I^{(1)}=\bigcup_{j=1}^{r}\left\{ s+\sum_{\ell=0}^{j-1}a_{\ell}^{(1)}+\sum_{\ell=0}^{j-2}(a'_{\ell}+1):s=1,2,\dots,a'_{j-1}+1\right\} ,\]
where the second sum is zero for $j=1$, and\[
I'=\{ i_{t}-\sum_{\ell=0}^{t-1}a_{\ell}^{(1)}:t=1,2,\dots,r\},\]
where $I=\{ i_{1},i_{2},\dots,i_{r}\}$. Observe that $i_{\ell}^{(1)}=i_{t}$
for $\ell=t+\sum_{s=0}^{t-1}a'_{s}=i'_{t}$. 

Assume first that $A(\widetilde{\mu_{1}})$ and $A(m')$ are true,
and let $\mathcal{E},\mathcal{F},\mathcal{G}$ be arbitrary flags
in a II$_{1}$ factor such that $\tau(E_{i})=\tau(F_{i})=\tau(G_{i})=i$
for $i=0,1,\dots,n$, and $\tau(1)=n$. Property $A(\widetilde{\mu_{1}})$
implies the existence of a projection $P_{1}\in\mathcal{A}$ such
that $\tau(P_{1})=r_{1}$ and \[
\tau(P_{1}\wedge E_{i_{\ell}^{(1)}})\ge\ell,\ell=1,2,\dots,r_{1}.\]
As noted above, we have $i_{\ell}^{(1)}=i_{t}$ for $\ell=i'_{t}$,
and therefore we have\[
\tau(P_{1}\wedge E_{i'_{t}})\ge t+\sum_{s=0}^{t-1}a'_{s}=i'_{t},\quad t=1,2,\dots,r,\]
with analogous inequalities for $\mathcal{F}$ and $\mathcal{G}$.
Consider now the factor $\mathcal{A}_{1}=P_{1}\mathcal{A}P_{1}$ with
the trace $\tau_{1}=\tau|\mathcal{A}_{1}$, so that $\tau_{1}(1_{\mathcal{A}'_{1}})=\tau(P_{1})=r_{1}$.
The inequalities above imply the existence of a flag $\mathcal{E}'$
in $\mathcal{A}_{1}$ such that $\tau_{1}(E'_{j})=j$ for $j=1,2,\dots,n_{1}$,
and\[
E'_{i'_{p}}\le P_{1}\wedge E_{i_{p}},\quad p=1,2,\dots,r.\]
 Analogous considerations lead to the construction of flags $\mathcal{F}'$
and $\mathcal{G}'$. Property $A(m')$ implies now the existence of
a projection $P\in\mathcal{A}_{1}$ such that $\tau_{1}(P)=r$,\[
\tau_{1}(P\wedge E_{i'_{p}}^{\prime})\ge p,\quad p=1,2,\dots,r,\]
and analogous inequalities are satisfied for $\mathcal{F}'$ and $\mathcal{G}'$.
Clearly the projection $P$ satisfies\[
\tau(P\wedge E_{i_{p}})\ge p,\quad p=1,2,\dots,r,\]
so that it solves the intersection problem for the sets $I,J,K$.

The case of property $B$ is settled analogously. The difference is
that $P_{1}$ is given as a lattice polynomial $P_{1}=p_{1}(\mathcal{E},\mathcal{F},\mathcal{G})$,
and the projections $E'_{j}$ can be taken to be of the form $P_{1}\wedge E_{i}$,
and hence they too are lattice polynomials in $\mathcal{E},\mathcal{F},\mathcal{G}$.
Finally, the solution $P$ is given as $P=p'(\mathcal{E}',\mathcal{F}',\mathcal{G}')$,
where the existence of $p'$ is given by property $B(m')$. One must
however assume that $\mathcal{E}',\mathcal{F}',\mathcal{G}'$ are
generic flags, and this simply amounts to an additional genericity
condition on the original flags.
\end{proof}
The preceding proposition shows that proving property $A(m)$ or $B(m)$
can be reduced to proving it for simpler measures, at least when $m$
is not extremal. A dual reduction is obtained by recalling that a
projection $P$ belongs to $S(\mathcal{E},I)$ if and only $P^{\perp}=1-P$
belongs to $S(\mathcal{E}^{\perp},I^{*})$. Moreover, if the sets
$I,J,K$ are associated to the measure $m\in\mathcal{M}_{r}$, then
$I^{*},J^{*},K^{*}$ are the sets associated to the measure $m^{*}$.
Therefore $A(m)$ is equivalent to $A(m^{*})$ and $B(m)$ is equivalent
to $B(m^{*})$.

To quantify these reductions, we define for each measure $m\in\mathcal{M}_{r}$
the positive integer $\kappa(m)$ as the number of gray parallelograms
in the puzzle obtained by inflating $m$. This is equal to the number
of white piece edges which have positive measure. Analogously, for
$m\in\mathcal{M}_{r}^{*}$, we define $\kappa^{*}(m)$ to be the number
of gray parallelograms in the puzzle obtained by {*}inflating $m$.
With this definition it is clear that \[
\kappa(m)=\kappa^{*}(m^{*}),\quad m\in\mathcal{M}_{r}.\]
Indeed, the two numbers count pieces of the same puzzle. 

With the notation of the preceding proposition, we have\[
\kappa(\widetilde{\mu_{1}})=\kappa(\mu_{1})<\kappa(m),\quad\kappa(m')<\kappa(m),\]
unless $m=\mu_{1}$. Indeed, $\kappa(\widetilde{\mu_{1}})=\kappa(\mu_{1})$
because $\mu_{1}$ and $\widetilde{\mu_{1}}$ are homologous, and
the supports of $\mu_{1}$ and $m'$ are strictly contained in the
support of $m$. In fact, the support of $m'$ does not contained
the root edges of $\mu_{1}$, and the support of $\mu_{1}$ does not
contain the root edges of the extremal summands of $m'.$ Thus the
preceding proposition also allows us to reduce the proof of these
properties to measures with smaller values of $\kappa$ in case either
$m$ or $m^{*}$ is not extremal. The exceptional situations in which
both $m$ and $m^{*}$ are extremal are very few in number. To see
this we need to use the structure of the convex polyhedral cone\[
C_{r}=\{\partial m:m\in\mathcal{M}_{r}\},\]
whose facets were determined in \cite{KTW}. If $\partial m=(\alpha,\beta,\gamma)$,
these facets are of two kinds. The first kind are the \emph{chamber
facets} determined by an equality of the form $\alpha_{\ell}=\alpha_{\ell+1}$,
$\beta_{\ell}=\beta_{\ell+1}$, $\gamma_{\ell}=\gamma_{\ell+1}$ for
$1\le\ell<r$ or $\alpha_{r}=\omega(m)$, $\beta_{r}=\omega(m)$,
$\gamma_{r}=\omega(m)$. The second kind are the \emph{regular facets}
determined by Horn identities\[
\sum_{i\in I}\alpha_{i}+\sum_{j\in J}\beta_{j}+\sum_{k\in K}\gamma_{k}=\omega(m),\]
where $I,J,K\subset\{1,2,\dots,r\}$ have $s<r$ elements and $c_{IJK}=1$.

For a given measure $m\in\mathcal{M}_{r}$, we define the number of
attachment points $\Gamma(m)$ to be the number of chamber facets
to which $m$ does not belong. The reason for this terminology is
that $\Gamma(m)$ is precisely the number of points on the sides of
$\triangle_{r}$ which are endpoints of interior edges in the support
of $m.$ The vertices of $\triangle_{r}$ should also counted as attachment
points when they are branch points of the measure.

\begin{prop}
Let $m\in\mathcal{M}_{r}$ be an extremal rigid measure. If $m^{*}$
is extremal as well, then $\Gamma(m)=1$.
\end{prop}
\begin{proof}
Assume that $m$ and $m^{*}$ are both extremal, and $\Gamma(m)>1$.
Note first that $\partial m$ is extremal in $C_{r}$. Indeed, in
the contrary case, we would have $\partial m=\partial m_{1}+\partial m_{2}$
with $\partial m_{1}$ not a positive multiple of $\partial m$. This
would however imply $m=m_{1}+m_{2}$ by rigidity, and hence $m_{1}$
is a multiple of $m$, a contradiction. 

Next, since $\Gamma(m)=\Gamma(m')$ for homologous $m,m'$, we may
assume that $m^{*}=\mu_{e}$ for some root edge $e$. Indeed, $m^{*}=m^{*}(e)\mu_{e}$
is homologous to $\mu_{e}$, and therefore $m$ is homologous to $\mu_{e}^{*}$.

The definition of $\Gamma(m)$ implies that $\partial m$ belongs
to $3r-\Gamma(m)=\dim C_{r}-\Gamma(m)$ chamber facets. However, an
extremal measure must belong to at least $\dim C_{r}-1$ facets, and
hence $\partial m$ belongs to at least one regular facet. As seen
earlier, there must then exist a clockwise overlay $(m_{1},m_{2})$
such that $m_{1}$ is obtained by contracting $m$. It follows from
Proposition \ref{prop:support-of-clockwise-measure} that $0\ne m_{1}^{*}\le m^{*}$.
Since $m_{1}^{*}$ has integer densities, we must have $m_{1}^{*}=m^{*}$,
and this implies that $m_{1}=m$, a contradiction.
\end{proof}
Thus the repeated application of the reduction procedure to $m$ and
$m^{*}$ leads eventually to one of the three skeletons pictured below.\begin{center}

\includegraphics[scale=0.7]{extras.eps}

\end{center}For these, the intersection problem is completely trivial. Indeed,
consider the first of the three on $\triangle_{r}$, and with $\omega(m)=s$.
We have then $I=\{1,2,\dots,r\}$ and $J=K=\{ s+1,s+2,\dots,s+r\}$,
and the desired element in \[
S(\mathcal{E},I)\cap S(\mathcal{F},J)\cap S(\mathcal{G},K)\]
 is simply $E_{r}$. Thus $A(m)$ is true for this measure. To show
that $B(m)$ is true as well, we must verify that generically we also
have\[
\tau(E_{r}\wedge F_{\ell})=\tau(E_{r}\wedge G_{\ell})=\max\{0,r+\ell-n\}\]
for $\ell=1,2,\dots n$. This follows easily from the following result.

\begin{prop}
\label{pro:Zariski}Let $E$ and $F$ be two projections in a finite
factor $\mathcal{A}$. There is an open dense set $O\subset\mathcal{U}(\mathcal{A})$
such that\[
\tau(E\wedge UFU^{*})=\max\{0,\tau(E)+\tau(F)-\tau(1)\}\]
for $U\in O$.
\end{prop}
\begin{proof}
Replacing $E$ and $F$ by $E^{\perp}$ and $F^{\perp}$ if necessary,
we may assume that $\tau(E)+\tau(F)\le\tau(1)$. Since $\mathcal{A}$
is a factor, we can replace $F$ with any other projection with the
same trace. In particular, we may assume that $F\le E^{\perp}$. The
condition $\tau(E\wedge UFU^{*})=0$ is satisfied if the operator
$FUF$ is invertible on the range of $F$. The proposition follows
because the set $O$ of unitaries satisfying this condition is a dense
open set in $\mathcal{U}(\mathcal{A})$. To verify this fact, it suffices
to consider the case in which the algebra $\mathcal{A}$ is of the
form $\mathcal{A}=\mathcal{B}\otimes M_{2}(\mathbb{C})$ for some
other finite factor $\mathcal{B}$, and \[
F=\left[\begin{array}{cc}
1 & 0\\
0 & 0\end{array}\right].\]
An arbitary unitary $U\in\mathcal{A}$ can be written as\[
U=\left[\begin{array}{cc}
T & (I-TT^{*})^{1/2}W\\
V(I-T^{*}T)^{1/2} & -VT^{*}W\end{array}\right],\]
where $T,V,W\in\mathcal{B}$, $V$ and $W$ are unitary, and $\| T\|\le1$.
Since $\mathcal{B}$ is a finite von Neumann algebra, $T$ can be
approximated arbitrarily well in norm by an invertible operator $T'$,
in which case $U$ is approximated in norm by the operator\[
U'=\left[\begin{array}{cc}
T' & (I-T'T^{\prime*})^{1/2}W\\
V(I-T^{\prime*}T')^{1/2} & -VT^{\prime*}W\end{array}\right]\]
with $FU'F$ invertible. In finite dimensions, the complement of $O$
is defined by the single homogeneous polynomial equation $\det(FUF+F^{\perp})=0$.
Thus $O$ is open in the Zariski topology.
\end{proof}
\begin{cor}
Properties $A(m)$ and $B(m)$ are true for all rigid measures $m$.
\end{cor}
This proves finally Theorems \ref{th:intersections-not-generic} and
\ref{th:intersections-generic-case}. The fact that Theorem \ref{th:the-Horn-inequalities-in-any-A}
follows from Theorem \ref{th:intersections-not-generic} was already
shown in \cite{Ber-Li-Romega}.

\section{Some Illustrations}

We have just seen that proving property $A(I,J,K)$ or $B(I,J,K)$
can be reduced, in case $c_{IJK}=1$, to the case in which the associated
measure $m$ has precisely one attachment point. We will illustrate
how this reduction works in a few cases.

Given a measure $m\in\mathcal{M}_{r}$, a point $A_{\ell}$, $\ell=1,2,\dots,r$,
is an attachment point of $m$ precisely when $m(A_{\ell}X_{\ell})>0$.
The solution to the associated Schubert intersection problem will
only depend on the projections $E_{i(\ell)}$ where $\ell$ is an
attachment point. These projections, and the analogous $F_{j(\ell)},G_{k(\ell)},$
will be called the \emph{attachment projections} for the problem.
With the notation Proposition \ref{prop:reduction-of-AB(m)}, the
attachment projections of $\widetilde{\mu_{1}}$ are exactly the same
as those of $\mu_{1}$, and are therefore among the attachment projections
of $m$. The attachment projections of $m^{*}$ are of the form $P^{\perp}=1-P$,
where $P$ is an attachment projection for $m$. These observations
allow us to construct solutions to intersection problems without actually
having to construct the measure $\widetilde{\mu_{1}}$ and focus instead
on the attachment projections of $\mu_{1}$.

We proceed now to solve the intersection problems associated with
some skeletons. Consider first an extreme measure $m\in\mathcal{M}_{r}$
with two attachment points. The following picture shows the supports
of $m$ and $m^{*}$.\begin{center}

\includegraphics[scale=1]{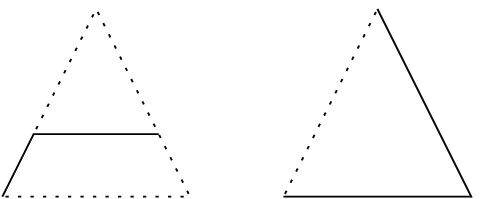}

\end{center}For the illustration we took $r=3$ and density 3 on the support,
but the results will hold for the general case. Note that $m^{*}$
is a sum of two extremal measures with one attachment point each.
If $X$ and $Z$ are the attachment projections of $m$, the attachment
projections of these skeletons are $X^{\perp}$ and $Z^{\perp}$.
Neither of the two skeletons precedes the other, and following the
method of Proposition \ref{prop:reduction-of-AB(m)}, we see that
the solution of the intersection problem associated with $m^{*}$
is generically $X^{\perp}\wedge Z^{\perp}$. It follows that the intersection
problem associated with $m$ has the generic solution $X\vee Y$.

There are two kinds of skeletons with three attachment points. The
first one, and its dual, are illustrated below.\begin{center}

\includegraphics[scale=1]{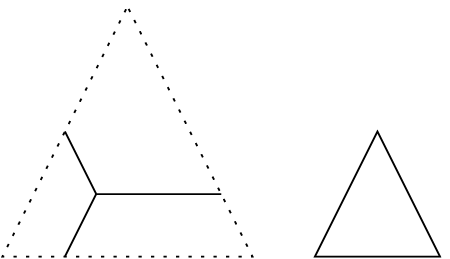}

\end{center}Assume that the attachment projections are $X,Y$ and $Z$. As in
the preceding situation, $m^{*}$ is a sum of three extremal measures
with one attachment point, and there are no precedence relations among
the skeletons. It follows that the generic solution of the intersection
problem is $X\vee Y\vee Z$.

The two cases just mentioned correspond to the reductions considered
in \cite{th-th} for finite dimensions, and in \cite{CoDy-reduction}
for the factor case. Note however that these papers also apply these
reductions when $c_{IJK}>1$.

Consider next the other kind of skeleton with three attachment points,
and with attachment projections $X,Y,Z$.\begin{center}

\includegraphics[scale=1]{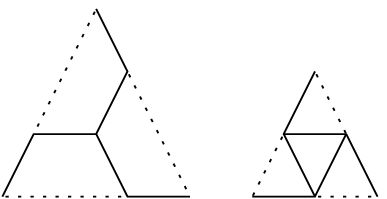}

\end{center}In this case, $m^{*}$ is the sum of three extremal measures with
two attachment points each, and with no precedence relations. The
intersection problems associated with the three skeletons have then
generic solutions $X^{\perp}\vee Y^{\perp},X^{\perp}\vee Z^{\perp}$,
and $Y^{\perp}\vee Z^{\perp}$. According to Proposition \ref{prop:reduction-of-AB(m)},
the solution of the intersection problem for $m^{*}$ will be (generically)
the intersection of these three projections, so that the problem associated
with $m$ has the solution\[
(X\wedge Y)\vee(X\wedge Z)\vee(Y\wedge Z).\]

Several of the proofs of Horn inequalities in the literature can now
be deduced by considering rigid measures which are sums of extremal
measures with 1,2 or 3 attachment points. Consider, for instance,
a measure $m\in\mathcal{M}_{r}$ defined by \[
m=\rho+\sum_{\ell=1}^{r}(\mu_{\ell}+\nu_{\ell}),\]
where $\rho$ has attachment point $C_{r}$, $\mu_{1}$ has attachment
point $A_{r}$, $\nu_{1}$ has attachment point $B_{r}$, $\mu_{\ell}$
has attachment points $A_{r-\ell+1}$ and $C_{\ell-1}$, and $\nu_{\ell}$
has attachment points $B_{r-\ell+1}$ and $C_{\ell-1}$ for $\ell>1$.\begin{center}

\includegraphics[scale=1]{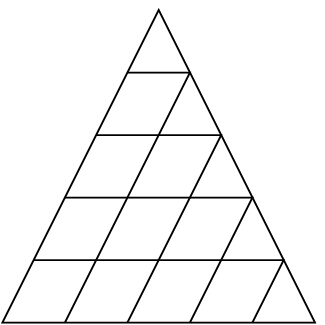}

\end{center}The only precedence relations are $\mu_{\ell}\prec_{0}\nu_{k}$ and
$\nu_{\ell}\prec_{0}\mu_{k}$ for $\ell<k$. Generically, the associated
intersection problem is solved as follows. Set $P_{0}=G_{r}$ and\[
P_{\ell+1}=[(G_{\ell}\wedge P_{\ell})\vee(F_{r-\ell}\wedge P_{\ell})]\wedge[(G_{\ell}\wedge P_{\ell})\vee(E_{r-\ell}\wedge P_{\ell})]\]
 for $\ell=1,2,\dots,r-1$. The space $P_{r}$ is the generic solution.
The sets $I,J,K$ associated with $m$ are easily calculated. Using
the notations\[
c=\omega(\rho),\quad a_{\ell}=\omega(\mu_{\ell}),\quad b_{\ell}=\omega(\nu_{\ell})\quad{\rm for}\;\ell=1,2,\dots,r,\]
we have \[
n=r+c+\sum_{\ell=1}^{r}(a_{\ell}+b_{\ell}),\]
and $I=\{ n+1-(a_{1}+a_{2}+\cdots+a_{\ell}+\ell):\ell=1,2,\dots,r\}$,
$J=\{ n+1-(b_{1}+b_{2}+\cdots+b_{\ell}+\ell):\ell=1,2,\dots,r\}$,
and $K=\{ a_{1}+b_{1}+\cdots+a_{\ell}+b_{\ell}+\ell:\ell=1,2,\dots,r\}$.
These sets yield the eigenvalue inequalities proved in \cite{tho-free1}.

Consider next sequences of integers\[
0\le z_{1}\le z_{2}\le\cdots\le z_{p},\quad0\le w_{1}\le w_{2}\le\cdots\le w_{p}\]
such that $z_{p}+w_{p}\le r$, and consider the measure $m\in\mathcal{M}_{r}$
defined by \[
m=\sum_{i=1}^{p}\mu_{i},\]
 where $\mu_{\ell}$ has attachment points $A_{z_{\ell}},B_{w_{\ell}}$,
and $C_{r-z_{\ell}-w_{\ell}}$.\begin{center}

\includegraphics[scale=1]{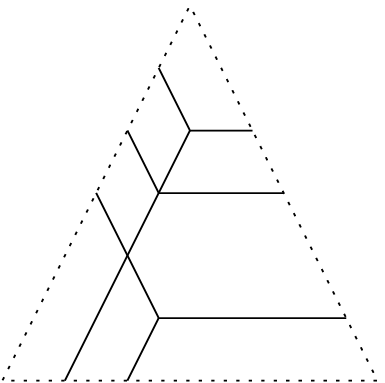}

\end{center}The illustration uses $p=3$, $r=6$, $z_{1}=1,$$z_{2}=2$, $z_{3}=3$,
$w_{1}=w_{2}=1$, and $w_{3}=2$. We have $\mu_{\ell}\prec_{0}\mu_{k}$
only when $\ell<k$, $w_{\ell}<w_{k}$ and $z_{\ell}<z_{k}$. If we
set $P_{1}=E_{z_{1}}\vee F_{w_{1}}\vee G_{r-z_{1}-w_{1}}$ and\[
P_{\ell+1}=(E_{z_{\ell+1}}\wedge P_{\ell})\vee(F_{w_{\ell+1}}\wedge P_{\ell})\vee(G_{r-z_{\ell+1}-w_{\ell+1}}\wedge P_{\ell})\]
for $\ell=1,2,\dots,d-1$, then $P_{d}$ is the generic solution of
the intersection problem. Assume that $\omega(\mu_{i})=1$ for all
$i$, and use the notation \[
1_{x<y}=\begin{cases}
1 & {\rm if}\: x<y,\\
0 & {\rm if}\: x\ge y.\end{cases}\]
 Then for the corresponding intersection problem we have $n=r+p$,
$I(\ell)=\ell+\sum_{i=1}^{p}1_{z_{i}<\ell}$, $J(\ell)=\ell+\sum_{i=1}^{p}1_{w_{i}<\ell}$,
and $n+1-K(r+1-\ell)=\ell+\sum_{i=1}^{p}1_{w_{i}+z_{i}<\ell}$ for
$\ell=1,2,\dots,r$. These sets yield the eigenvalue inequalities
proved in \cite{tho-free2}.

One can produce such families of inequalities using more complicated
skeletons. Observe for instance that, given integers $a,b,c,d$ such
that $a+b+c+d=r$, there exists a skeleton in $\triangle_{r}$ with
attachment points $A_{a},A_{a+b+c},B_{b+d}$, and $C_{c+d}$. Call
$\mu_{a,b,c,d}$ the smallest extremal measure with integer densities
supported by this skeleton. A measure of the form\[
m=\sum_{\ell=1}^{p}\mu_{a_{\ell},b_{\ell},c_{\ell},d_{\ell}}\]
 will be rigid if the following conditions are satisfied:\[
a_{\ell}\le a_{\ell+1},d_{\ell}\le d_{\ell+1},c_{\ell}+d_{\ell}\le c_{\ell+1}+d_{\ell+1},b_{\ell}+d_{\ell}\le b_{\ell+1}+d_{\ell+1}\]
for $\ell=1,2,\dots,p-1$. Moreover, $\mu_{a_{\ell},b_{\ell},c_{\ell},d_{\ell}}\prec\mu_{a_{\ell'},b_{\ell'},c_{\ell'},d_{\ell'}}$
implies $\ell\le\ell'$. The corresponding intersection problem will
be solved by dealing successively with these summands. The reader
will have no difficulty writing out the sets $I,J,K\subset\{1,2,\dots,n\}$,
where $n=r+2p$. The following figure illustrates the case $p=2$
with $r=8$, $a_{1}=d_{1}=1$, $b_{1}=c_{1}=3$, $a_{2}=2$, $b_{2}=c_{2}=1$,
and $d_{2}=4$.\begin{center}

\includegraphics[scale=.7]{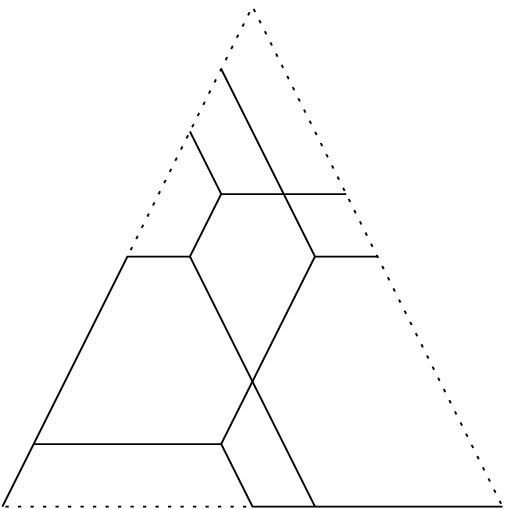}

\end{center}

We deal next with a somewhat more complicated extremal measure, whose
support has the shape pictured below along with its dual.\begin{center}

\includegraphics[scale=.65]{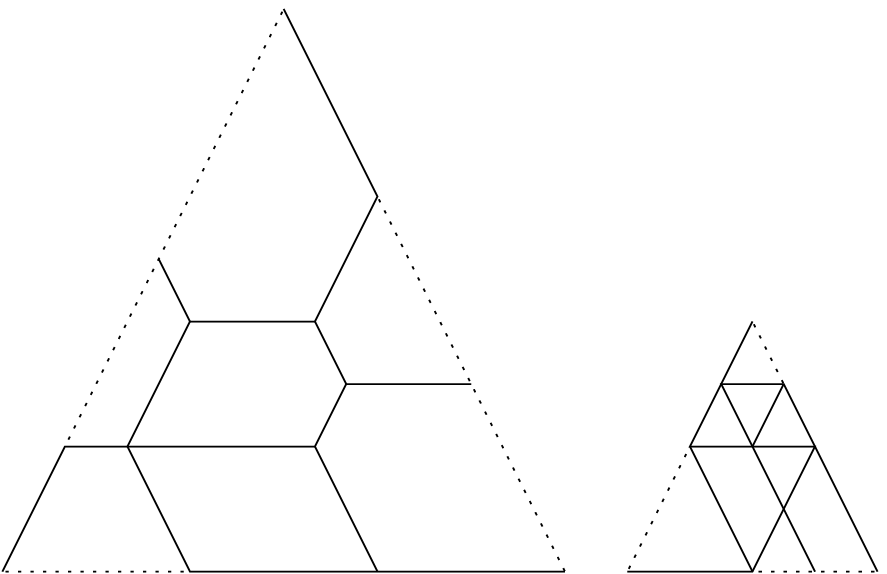}

\end{center}Denote the attachment projections on the $A$ side by $X_{1}\le X_{2}$
, on the $B$ side by $Y_{1}\le Y_{2}$, and on the $C$ side $Z_{1}\le Z_{2}$.
In the illustration we used the measure $m$ which assigns unit mass
to the root edges of the skeleton, and this measure has weight $\omega(m)=4$.
The measure $m^{*}$ is a sum of six extremal measures with supports
pictured below.\begin{center}

\includegraphics[scale=.65]{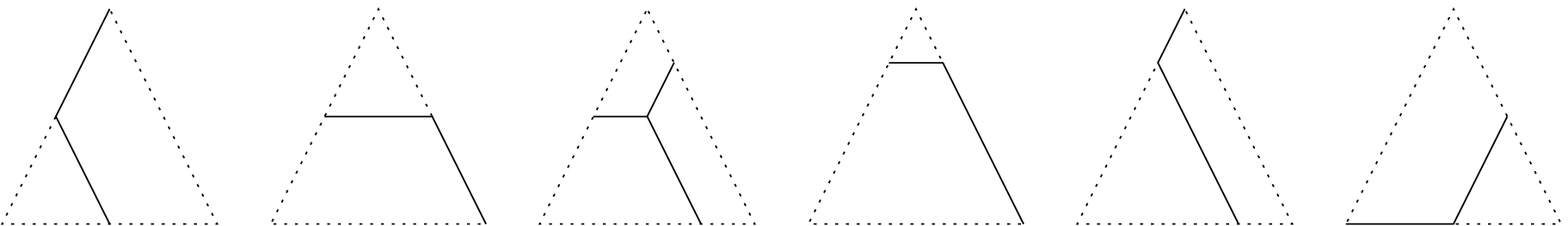}

\end{center}Denote the corresponding measures in $\mathcal{M}_{4}^{*}$ by $\mu_{\ell}$
with $\ell=1,2,\dots,6.$ the attachment projections for these measures
are easily found. For instance, $\mu_{1}$ has attachment projections
$X_{2}^{\perp},Y_{1}^{\perp}$, and $\mu_{3}$ has attachment projections
$X_{2}^{\perp},Y_{2}^{\perp}$, and $Z_{2}^{\perp}$. Recalling that
clocks run backwards in $\mathcal{M}^{*}$, we easily determine that\[
\mu_{1}\prec_{0}\mu_{4}\prec_{0}\mu_{6},\quad\mu_{2}\prec_{0}\mu_{5}\prec_{0}\mu_{6},\quad\mu_{3}\prec_{0}\mu_{6},\]
and no other direct comparisons occur. It is now easy to see that
the generic solution is obtained as follows. Form first the projection\[
P_{1}=(X_{2}\wedge Y_{1})\vee(X_{2}\wedge Z_{2})\vee(X_{2}\wedge Y_{2}\wedge Z_{1})\]
corresponding with the measure $\mu_{1}+\mu_{2}+\mu_{3}$. Next calculate\[
P_{2}=[(X_{1}\vee P_{1})\wedge(Z_{1}\wedge P_{1})]\vee[(X_{1}\vee P_{1})\wedge(Y_{2}\wedge P_{1})]\]
corresponding with $\mu_{4}+\mu_{5}$. Finally, the solution is\[
P=(Y_{2}\vee P_{2})\wedge(Z_{2}\wedge P_{2}).\]

The examples above illustrate the fact that passing from an extremal
measure to its dual yields a dramatic simplification of the intersection
problem. We offer, mostly to further illustrate this point, an example
of a rather complicated skeleton. The reader will easily identify
15 skeletons in the dual picture.\begin{center}

\includegraphics[scale=0.7]{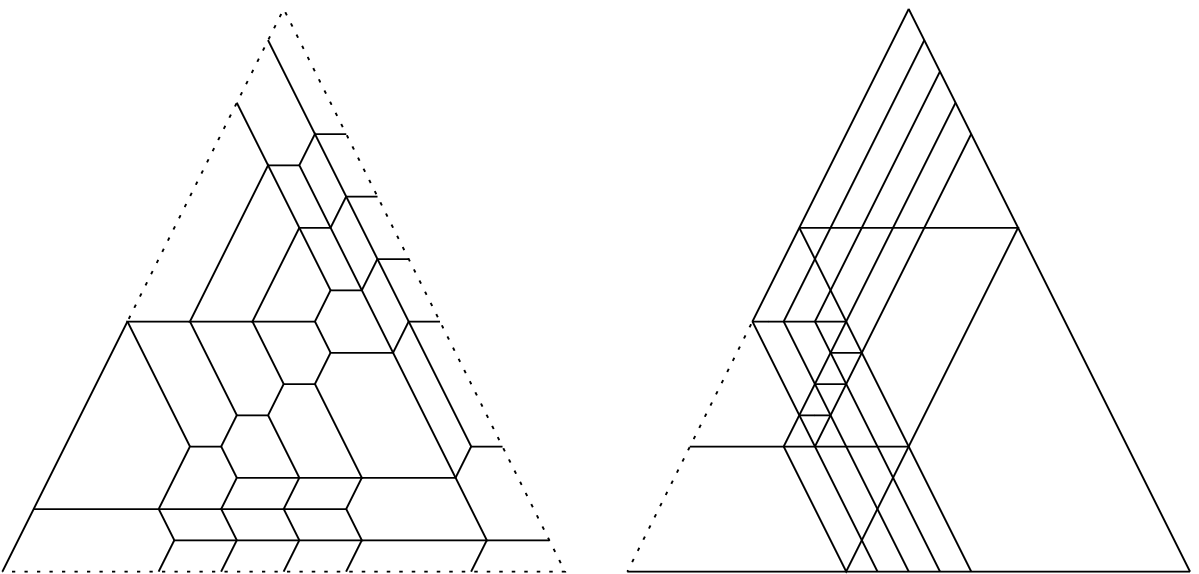}

\end{center}The number 15 is, not coincidentally, the number of attachment points
of the skeleton.

\section{Connection to Invariant Subspaces}

The smallest example of sets $I,J,K$ with $c_{IJK}>1$ is given by
$I=J=K=\{2,4,6\}\subset\{1,2,3,4,5,6\}$. Fix a II$_{1}$ factor $A$
with trace normalized so that $\tau(1)=2$, and fix an element $T\in\mathcal{A}$.
We will view $\mathcal{A}$ as an algebra of operators on a Hilbert
space $H$, and consider the factor $\mathcal{B}=\mathcal{A}\otimes M_{3}(\mathbb{C})$
acting on $H\oplus H\oplus H$. That is, $B$ consists of all operator
matrices $[T_{ij}]_{i,j=1}^{3}$ with $T_{ij}\in\mathcal{A}$, with
trace defined by\[
\tau([T_{ij}]_{i,j=1}^{3})=\sum_{j=1}^{3}\tau(T_{jj}).\]
We construct now the following spaces:\begin{eqnarray*}
X_{2} & = & \{\xi\oplus0\oplus0:\xi\in H\},\\
Y_{2} & = & \{0\oplus\xi\oplus0:\xi\in H\},\\
Z_{2} & = & \{0\oplus0\oplus\xi:\xi\in H\},\\
X_{4} & = & \{\xi\oplus\eta\oplus\eta:\xi,\eta\in H\},\\
Y_{4} & = & \{\eta\oplus\xi\oplus\eta:\xi,\eta\in H\},\\
Z_{4} & = & \{\eta\oplus T\eta\oplus\xi:\xi,\eta\in H\}.\end{eqnarray*}
It is easy to see that the orthogonal projections $E_{2},F_{2},G_{2},E_{4},F_{4},G_{4}$
onto these spaces belong to $A$ and \[
\tau(E_{j})=\tau(F_{j})=\tau(G_{j})=j,\quad j=2,4.\]
Indeed, we can write these projections explicitly:

\[
E_{2}=\left[\begin{array}{ccc}
1 & 0 & 0\\
0 & 0 & 0\\
0 & 0 & 0\end{array}\right],F_{2}=\left[\begin{array}{ccc}
0 & 0 & 0\\
0 & 1 & 0\\
0 & 0 & 0\end{array}\right],G_{2}=\left[\begin{array}{ccc}
0 & 0 & 0\\
0 & 0 & 0\\
0 & 0 & 1\end{array}\right],\]
\[
E_{4}=\left[\begin{array}{ccc}
1 & 0 & 0\\
0 & \frac{1}{2} & \frac{1}{2}\\
0 & \frac{1}{2} & \frac{1}{2}\end{array}\right],F_{4}=\left[\begin{array}{ccc}
\frac{1}{2} & 0 & \frac{1}{2}\\
0 & 1 & 0\\
\frac{1}{2} & 0 & \frac{1}{2}\end{array}\right],\]
\[
G_{4}=\left[\begin{array}{ccc}
(1+T^{*}T)^{-1} & (1+T^{*}T)^{-1}T^{*} & 0\\
T(1+T^{*}T)^{-1} & T(1+T^{*}T)^{-1}T^{*} & 0\\
0 & 0 & 1\end{array}\right].\]
The trace of $G_{4}$ is seen to be $4$ because $\left[\begin{array}{cc}
(1+T^{*}T)^{-1} & (1+T^{*}T)^{-1}T^{*}\\
T(1+T^{*}T)^{-1} & T(1+T^{*}T)^{-1}T^{*}\end{array}\right]$ is the range projection of the partial isometry $\left[\begin{array}{cc}
(1+T^{*}T)^{-1/2} & 0\\
T(1+T^{*}T)^{-1/2} & 0\end{array}\right]$ which has initial projection $\left[\begin{array}{cc}
1 & 0\\
0 & 0\end{array}\right]$. Assume that $P\in S(\mathcal{E},I)\cap S(\mathcal{F},J)\cap S(\mathcal{G},K)$.
In other words, $\tau(P)=3$, $\tau(P\wedge E_{2})\ge1$,$\tau(P\wedge F_{2})\ge1$,$\tau(P\wedge G_{2})\ge1$,$\tau(P\wedge E_{4})\ge2$,$\tau(P\wedge F_{4})\ge2$,
and $\tau(P\wedge G_{4})\ge2$. It follows then that there exist projections
$Q,Q',Q''\in A$ such that $\tau(Q)\ge1$, $\tau(Q')\ge1$, $\tau(Q'')\ge1$,
and $P\ge Q\oplus Q'\oplus Q''$, which implies that $\tau(Q)=\tau(Q')=\tau(Q'')=1$
and $P=Q\oplus Q'\oplus Q''$. Next observe that\[
P\wedge E_{4}=Q\oplus\left[(Q'\oplus Q'')\wedge\left[\begin{array}{cc}
\frac{1}{2} & \frac{1}{2}\\
\frac{1}{2} & \frac{1}{2}\end{array}\right]\right]=Q\oplus\left[\begin{array}{cc}
\frac{1}{2}Q'\wedge Q'' & \frac{1}{2}Q'\wedge Q''\\
\frac{1}{2}Q'\wedge Q'' & \frac{1}{2}Q'\wedge Q''\end{array}\right].\]
This projection must have trace at least $2$, and therefore $Q'=Q''$.
Analogously, the condition $\tau(P\wedge F_{4})\ge2$ implies that
$Q=Q''$. We conclude that $P=Q\oplus Q\oplus Q$. Finally, $\tau(P\wedge G_{4})\ge2$
will imply that $QTQ=TQ$, so that $Q$ is an invariant projection
for the operator $T$. Thus the solution of this particular intersection
problem implies the existence of invariant projections of trace $1$
for every $T\in\mathcal{A}$. In \cite{CoDy-reduction} it is shown
that this problem has an approximate solution. More precisely, given
$\varepsilon>0$, there exist projections $Q,Q_{1}\in\mathcal{A}$
such that $\tau(Q)=1$, $Q\le Q_{1}$, $\tau(Q_{1})<1+\varepsilon$,
and $Q_{1}TQ=TQ$. This leads to an approximate solution of the intersection
problem. One would expect that solving the intersection problem for
more complicated sets with $c_{IJK}>1$ would require considerable
progress in the study of II$_{1}$ factors.

\section{Applications of Free Probability}

In this brief section we give two applications of free products of
von Neumann algebras and free probability. First, we show that all
finite von Neumann algebras with a normal, faithful trace admit a
trace-preserving embedding into a factor of type II$_{1}$. This completes
the proof of the Horn inequalities for selfadjoint elements in such
algebras.

\begin{prop}
Let $\mathcal{A}_{j}$ be von Neumann algebras equipped with normal,
faithful, tracial states $\tau_{j}$, $j=1,2,$ and let $(\mathcal{A},\tau)=(\mathcal{A}_{1},\tau_{1})*(\mathcal{A}_{2},\tau_{2})$
be the free product von Neumann algebra. If $\mathcal{A}_{2}$ is
diffuse, i.e., it has no minimal projections, and $\mathcal{A}_{1}$
is not a copy of the complex numbers, then $\mathcal{A}$ is a ${\rm II}{}_{1}$
factor.
\end{prop}
\begin{proof}
Let $\mathcal{B}$ be the C{*}-subalgebra of $\mathcal{A}$ generated
by the union of the copies of $\mathcal{A}_{1}$ and $\mathcal{A}_{2}$
in $\mathcal{A}$. Then $\mathcal{B}$ is isomorphic to the C{*}-algebra
reduced free product of $(\mathcal{A}_{1},\tau_{1})$ and $(\mathcal{A}_{2},\tau_{2})$,
and it is dense in $\mathcal{A}$ in the strong operator topology.
By Proposition 3.2 of \cite{ken-stable-rank}, $\mathcal{B}$ has
a unique tracial state. It follows that $\mathcal{A}$ has a unique
normal tracial state. As $\mathcal{A}$ is clearly infinite dimensional,
it is a II$_{1}$ factor. 
\end{proof}
Next, we will argue that arbitrary projections in a factor of type
II$_{1}$ can be perturbed into general position by letting one of
them evolve according to free Brownian motion. This perturbation will
take place typically in a larger factor obtained as a free product,
with the free Brownian motion in one of the factors.

Let $\mathcal{A}$ be a II$_{1}$ factor with trace $\tau$, and let
$P,Q\in\mathcal{A}$ be two projections. Let $U_{t}$ be a free right
unitary Brownian motion, free from $\{ P,Q\}$. Recall that a free
right unitary Brownian motion is the solution of the free stochastic
differential equation \[
U_{0}=1,\quad dU_{t}=iU_{t}\, dX_{t}-\frac{1}{2}U_{t}t\, dt,\]
where $X_{t}$ is a standard additive free Brownian motion (cf. \cite{biane}).
For our purposes, the following three properties of a unitary Brownian
motions are crucial:

\begin{enumerate}
\item $t\mapsto U_{t}$ is norm-continuous;
\item for any $\varepsilon>0$, $U_{t}^{*}U_{t+\varepsilon}$ is free from
$U_{s}$ for all $s<t$.
\end{enumerate}
For the purposes of the following result, we will say that $P$ and
$Q$ are in general position if $\tau(P\wedge Q)=\max\{0,\tau(P)+\tau(Q)-1\}$.

\begin{thm}
The projections $U_{t}PU_{t}^{*}$ and $Q$ are in general position
for every $t>0$.
\end{thm}
\begin{proof}
Fix $t>0$, and set $P_{t}=U_{t}PU_{t}^{*}$. As in the proof of Proposition
\ref{pro:Zariski}, we may and shall assume that and $\tau(P)+\tau(Q)\le1$.
Arguing by contradiction, assume that $P_{t}\wedge Q\ne0$. Setting
$R=P_{t}\wedge Q$, observe that the function\[
f(s)=\tau((RP_{s}R-R)^{2}),\quad s\ge0,\]
is nonnegative, and therefore $f$ has a minimum at $s=t$. The fact
that $U_{t}$ is a free Brownian motion, and Ito calculus, imply that
$f$ is a differentiable function, and \[
f'(t)=\tau(R)(-1+\tau(P)+\tau(R)-\tau(R)\tau(P)),\]
where we used the fact that $(RP_{t}R)^{2}=RP_{t}R=R$. Now, we have
$0<\tau(R)\le\tau(P)$ and $1-\tau(P)\ge\tau(R)$, so that this relation
implies\[
f'(t)\le-\tau(R)^{2}\tau(P)<0.\]
This however is not compatible with $f(t)$ being a minimum.
\end{proof}
If $\mathcal{E}$ and $\mathcal{F}$ are two flags in $\mathcal{A}$,
the preceding result yields a unitary $U$, arbitarrily close to $1$,
so that the spaces of the flag $U\mathcal{F}U^{*}$ are in general
position relative to the spaces of $\mathcal{E}$. Dealing with three
flags would require the use of two Brownian motions, free from each
other and from the flags. In order to obtain flags which are generic
for a given intersection of three Schubert cells, this construction
would have to be iterated following the inductive procedure of Proposition
\ref{prop:reduction-of-AB(m)}.

\end{document}